\documentclass[sn-mathphys,Numbered]{sn-jnl}

\usepackage{graphicx}%
\usepackage{multirow}%
\usepackage{amsmath,amssymb,amsfonts}%
\usepackage{amsthm}%
\usepackage{mathrsfs}%
\usepackage[title]{appendix}%
\usepackage{xcolor}%
\usepackage{textcomp}%
\usepackage{manyfoot}%
\usepackage{booktabs}%
\usepackage{algorithm}%
\usepackage{algorithmicx}%
\usepackage{algpseudocode}%
\usepackage{listings}%
\usepackage{orcidlink}%
\newcommand{\Id}{\mathrm{Id}}
\newcommand{\Vol}{\mathrm{Vol}}
\numberwithin{equation}{section}

\theoremstyle{thmstyleone}%
\newtheorem{theorem}{Theorem}[section]
\theoremstyle{thmstyletwo}%
\newtheorem{remark}{Remark}[section]%
\newtheorem{lemma}[theorem]{Lemma}

\theoremstyle{thmstylethree}%
\newtheorem{definition}{Definition}[section]%

\begin{document}

\title[Spectral Distribution of Twisted Laplacian]{Spectral Distribution of Twisted Laplacian on Typical Hyperbolic Surfaces of High Genus}


\author*{\fnm{Yulin} \sur{Gong} ~\orcidlink{0009-0009-9793-1883}}\email{gongyl22@mails.tsinghua.edu.cn}

\affil{\orgdiv{Department of Mathematical Sciences}, \orgname{Tsinghua University}, \city{Beijing}, \postcode{100084}, \country{China}}


\abstract{We investigate the spectral distribution of the twisted Laplacian associated with uniform square-integrable bounded harmonic 1-form on typical hyperbolic surfaces of high genus. First, we estimate the spectral distribution by the supremum norm of the corresponding harmonic form. Subsequently, we show that the square-integrable bounded harmonic form exhibits a small supremum norm for typical hyperbolic surfaces of high genus. Based on these findings, we prove a uniform Weyl law for the distribution of real parts of the spectrum on typical hyperbolic surfaces.}

\keywords{Spectral distribution, Twisted Laplacian, Harmonic form, Typical hyperbolic surfaces}



\maketitle

\section{Introduction}
In 1956, Selberg \cite{SA} proved the prime geodesic theorem (PGT) for the number of closed geodesics on a closed hyperbolic surface. The main tool in this classical paper is the Selberg trace formula (see \cite{BU, SA}), which connects the spectrum of the Laplacian and the period of closed geodesics. Later, Katsuda-Sunada \cite{KS}, Phillips-Sarnak \cite{PS}, Lalley \cite{LS1,LS2}, Babillot-Ledrappier \cite{BL}, Anantharaman \cite{AN1} studied the number of closed geodesics in a homology class on a hyperbolic surface, using the following twisted Laplacian associated with a harmonic 1-form:
\begin{definition}
  For any harmonic $1$-form $\omega$ on a closed hyperbolic surface $X$, the twisted Laplacian $\Delta_{\omega}$ is defined as follows: for any $f \in C^{\infty}(X)$ and any $x \in X$,
\begin{equation}\label{twistedlaplacian}
    \Delta_{\omega}f(x):=\Delta f(x)-2\langle \omega,df \rangle_{x}+|\omega|_{x}^{2}f(x).
\end{equation}
\end{definition}
Here $\Delta$ is negative Laplacian on $X$, and $|\omega|_{x}^{2}=\langle \omega,\omega \rangle_{x}$\footnote{Here $\langle-,-\rangle$ is $\mathbb{C}$-bilinear, i.e. $\langle \lambda \omega ,\mu \eta\rangle_{x}=\lambda \mu \langle \omega,\eta \rangle_{x}$ for any $\lambda, \mu \in \mathbb{C}$, and $\omega,\eta \in T^{*}_{x}X$.}. The twisted Laplacian has another definition: fix a point $o$ on the universal cover $\mathbb{H}$ of $X$, lift $\omega$ to $\mathbb{H}$, then for any $f \in C^{\infty}(\mathbb{H})$,
\begin{equation*}\label{univer}
  \Delta_{\omega}f=e^{\int_{o}^{x}\omega}\Delta(e^{-\int_{o}^{x}\omega}f).
\end{equation*}
The two definitions coincide when $f$ is $\Gamma$-invariant, where $\Gamma$ is the fundamental group of $X$. In this paper, we assume that $\omega \in \mathcal{H}^{1}(X,\mathbb{R})$, the space of real-valued harmonic $1$-forms, then $\Delta_{\omega}$ is a non-self-adjoint perturbation of Laplacian and has discrete spectrum on $L^{2}(X)$: 
\begin{equation}\label{equation}
\Delta_{\omega}\phi_{j}+\lambda_{j}^{\omega}(X)\phi_{j}=0, \ ||\phi_{j}||_{L^2}=1 \ \text{with} \  \lambda^{\omega}_{0}(X)<\Re \lambda^{\omega}_{1}(X) \leq \cdots.
\end{equation}
  Anantharaman \cite[Appendix]{AN1} proved that the principal eigenvalue $\lambda^{\omega}_{0}(X)$ is real and simple, and obtained the explicit value of $\lambda^{\omega}_{0}(X)$:
\begin{equation}\label{principalvalue}
  \lambda^{\omega}_{0}(X)=\mathrm{Pr}(\omega)(1-\mathrm{Pr}(\omega)),
\end{equation}
where the pressure $\mathrm{Pr}(\omega)$ (see \cite{KH})\footnote{Here $\omega$ is identified with $\omega(x,\xi)=\langle \omega,\xi \rangle_{x}$ for $(x,\xi) \in S^*X$.} satisfies 
\begin{equation}\label{estiofprin}
  1\leq \mathrm{Pr}(\omega)\leq 1+||\omega||_{L^{\infty}}, \ ||\omega||_{L^{\infty}}=\sup\limits_{x\in X}|\omega|_{x}.
\end{equation}
The \eqref{principalvalue} can also be obtained by Naud-Spilioti \cite[Proposition 2.3]{NS}. \\
\indent In this paper, we consider the spectrum of twisted Laplacian on typical hyperbolic surfaces of high genus. Monk \cite[Theorem 5]{ML2}, also see Monk \cite{ML1}, proved a uniform Weyl law for the usual Laplacian on typical surfaces of high genus. Recently, Monk-Stan \cite{MS23} proved the uniform Weyl law for Dirac spectrum of typical surfaces with any non-trivial spin structure. To describe her results, we introduce some notations.
\subsection{Uniform Weyl law for the usual Laplacian on random Weil-Petersson surfaces}
Let $\mathcal{M}_g$ be the moduli space of all closed hyperbolic surfaces of genus $g$, up to isometry. There is a natural probability measure $\mathbb{P}^{\mathrm{WP}}_{g}$ on the moduli space $\mathcal{M}_{g}$ (see \cite[Section 2]{MM} and \cite[Section 3.1.4.4]{ML1}). We say that a sequence of measurable events $\mathcal{E}_{g}\subset \mathcal{M}_{g}$ occurs with high probability if $\lim\limits_{g \to \infty}\mathbb{P}^{\mathrm{WP}}_{g}(\mathcal{E}_{g})=1$, and call a surface $X$ typical if $X \in \mathcal{E}_{g}$. We notice that for any two sequences of events $\mathcal{X}_{g}$ and $\mathcal{Y}_{g}$ which occur with high probability, $\mathcal{X}_{g} \cap \mathcal{Y}_{g}$ also occurs with high probability.\\
\indent Monk \cite{ML2} introduced the event $\mathcal{A}_{g}$ as follows: $X \in \mathcal{A}_{g}$ if
  \begin{equation}\label{geoassump}
    \mathrm{Inj}(X)\geq r_{g}=g^{-\frac{1}{24}}(\log g)^{\frac{9}{16}}, \quad \frac{\mathrm{Vol}(\{z\in X: \mathrm{Inj}_{z}(X)<\frac{1}{6}\log g\})}{\mathrm{Vol}(X)} \leq g^{-\frac{1}{3}},
  \end{equation}
   where $\mathrm{Inj}_{z}(X)$ is the injectivity radius at $z$ of $X$, and $\mathrm{Inj}(X)=\inf\limits_{z \in X}\mathrm{Inj}_{z}(X)$ is the injectivity radius of $X$. Recall that injectivity radius at a point $z$ of a Riemannian manifold $X$ is the largest radius for which the exponential map at $z$ is a diffeomorphism. By Mirzakhani \cite{MM} and Monk \cite{ML2}, $\lim\limits_{g \to \infty}\mathbb{P}^{\mathrm{WP}}_{g}(\mathcal{A}_{g})=1.$ Let $N^{P}_{X}(\lambda \in A)$ be the number of eigenvalues $\lambda$ of the operator $P$ on $X$ which satisfy $\lambda \in A$ and $\mathbf{1}_{A}$ be the indicator function of set $A$.
    \begin{theorem}[Monk, \cite{ML1, ML2}]\label{uniweyl}
      There is a universal constant $C_{0}>0$ and $g_{0}>0$ such that, for any $g>g_{0}$, any $0\leq a\leq b$ and any $X\in \mathcal{A}_{g},$ we have
      \begin{equation*}
          \frac{N_{X}^{-\Delta}(\lambda \in [a,b])}{\mathrm{Vol}(X)}=\frac{1}{4\pi}\int^{\infty}_{\frac{1}{4}} \mathbf{1}_{[a,b]}(\lambda)\tanh \left ( \pi\sqrt{\lambda-\frac{1}{4}} \right ) d\lambda+R(X,a,b),
      \end{equation*}
      where
      \begin{equation*}
          -C_{0}\sqrt{\frac{b+1}{\log g}}\leq R(X,a,b) \leq C_{0}\sqrt{\frac{b+1}{\log g}} \left[ \log \left(2+(b-a) \sqrt{\frac{\log g}{b+1}}\right) \right]^{\frac{1}{2}}.
      \end{equation*}
  \end{theorem}
\subsection{Main results} 
Schoen-Wolpert-Yau \cite{SWY} introduced the geometric quantity $\mathcal{L}_1(X)$ as follows:
\begin{equation}\label{lone}
\mathcal{L}_1(X):=\min \left\{\ell_{X}(\gamma) : 
\begin{array}{l}
\gamma=\gamma_1+\cdots+\gamma_k \text { is a simple closed } \\
\text { multi-geodesics separating } X
\end{array}\right\}
\end{equation}
This quantity was applied to investigate the first $2g-3$ small eigenvalues of the usual Laplacian on hyperbolic surfaces. Furthermore, the applications of $\mathcal{L}_1(X)$ are discussed in the works of Wu-Xue \cite{WX1, WX4}. Fix a function $\omega(g)$ satisfies $$\lim\limits_{g \to \infty}\omega(g)=\infty, \ \lim\limits_{g \to \infty}\frac{\omega(g)}{\log \log g}=0. $$ 
Nie-Wu-Xue \cite{NWX} defined
\begin{equation}\label{short}
  \mathcal{L}_{g}:=\{X \in \mathcal{M}_{g}: \ \mathcal{L}_1(X)\geq 2\log g-4\log \log g -\omega(g)\},
\end{equation}
and proved that $\lim\limits_{g \to \infty}\mathbb{P}^{\mathrm{WP}}_{g}(\mathcal{L}_g)=1$ in \cite[Theorem 6]{NWX}. \\
\indent Recall that $||\omega||_{L^{2}}^{2}=\int_{X}|\omega|_{x}^{2}d\mathrm{Vol}(x)$. Our main result is the following uniform Weyl law for twisted Laplacian $\Delta_{\omega}$ associated with uniform $L^{2}$ bounded $\omega \in \mathcal{H}^{1}(X,\mathbb{R})$.
\begin{theorem}\label{uniweylfortwist}
  There exists a universal constant $C>0$ such that for any $c>0$, there is $g_{c}>0$ depending only on $c$, for any $g>g_{c}$, any interval $I=[a,b]$ with $b\geq 0$, any $X \in \mathcal{A}_{g} \cap \mathcal{L}_{g}$, and any $\omega \in \mathcal{H}^{1}(X,\mathbb{R})$ with $||\omega||_{L^{2}}\leq c$, we have:
  \begin{equation}\label{maincount}
    \begin{aligned}
        \frac{N_{X}^{-\Delta_{\omega}}(\Re\lambda \in [a,b])}{\mathrm{Vol}(X)}=\frac{1}{4\pi}\int^{\infty}_{\frac{1}{4}} \mathbf{1}_{[a,b]}(\lambda)\tanh \left ( \pi\sqrt{\lambda-\frac{1}{4}} \right ) d\lambda+R_{\omega}(X,a,b),
    \end{aligned}
  \end{equation}    
  where
  \begin{equation}\label{maincountfin}
  -C\sqrt{\frac{b+1}{\log g}}\leq R_{\omega}(X,a,b) \leq C\sqrt{\frac{b+1}{\log g}} \left[ \log \left(2+(b-a) \sqrt{\frac{\log g}{b+1}}\right) \right]^{\frac{1}{2}}.
  \end{equation} 
\end{theorem}
Huber \cite{HUB} proved that the remainder term of the PGT is determined by the small eigenvalues of the usual Laplacian, and Wu-Xue \cite{WX3} investigated the remainder term of the PGT on typical surfaces. Monk \cite{ML1,ML2} and Wu-Xue \cite{WX2} obtained the distribution of small eigenvalues on a typical surface. For the spectrum of twisted Laplacian, we obtain the distribution of small eigenvalues $\lambda$ with $\Re \lambda<\frac{1}{4}$ on a typical surface by Wu-Xue \cite[Theorem 2]{WX2}. 
\begin{theorem}\label{uniweylfortwistsmalleigen}
    Fix an interval $I=[a,b]$ with $0\leq b <\frac{1}{4}$, then for any $0<\varepsilon<\sqrt{\frac{1}{4}-b}$ and $c>0$, there exists a sequence of events $\mathcal{N}_{g}^{b,\varepsilon}\subset \mathcal{M}_{g}$ defined in \eqref{smalleigencount} which occurs with high probability, and $g_{c,\varepsilon}>0$ such that for any $g>g_{c,\varepsilon}$, any $X \in \mathcal{A}_{g}\cap \mathcal{L}_{g} \cap \mathcal{N}_{g}^{b,\varepsilon}$ and any $\omega \in \mathcal{H}^{1}(X,\mathbb{R})$ with $||\omega||_{L^{2}}\leq c$, we have:
    \begin{equation}\label{lowcount1}
      N_{X}^{-\Delta_{\omega}}(\Re \lambda \in [a,b])\leq 1+g^{1-4\sqrt{\frac{1}{4}-b}+\varepsilon}.
    \end{equation}
\end{theorem}  
The spectral gap, namely $\lambda_{1}(X)$, is widely studied in geometry. For some $\alpha > 0$ and any $0<\varepsilon <\alpha$, we define a sequence of events $\mathcal{Q}^{\alpha, \varepsilon}_{g} \subset \mathcal{M}_{g}$ as follows:
\begin{equation}\label{gapevent}
\mathcal{Q}^{\alpha, \varepsilon}_{g}:=\left\{X \in \mathcal{M}_{g}: \lambda_{1}(X) > \alpha - \varepsilon \right\}.
\end{equation}
On random hyperbolic surfaces with the Weil-Petersson probabilistic model, we are concerned with the typical spectral gap $\alpha$ which satisfies for any $0<\varepsilon <\alpha$, $\lim\limits_{g \to \infty}\mathbb{P}^{\mathrm{WP}}_{g}(\mathcal{Q}^{\alpha, \varepsilon}_{g})=1$.\\
\indent Mirzakhani \cite{MM} first proved that $\alpha = \frac{1}{4} \left(\frac{\ln(2)}{2\pi + \ln(2)}\right)^2 \approx 0.002$ is a typical spectral gap. Two independent teams, Wu-Xue \cite{WX2} and Lipnowski-Wright \cite{LW}, significantly improved $\alpha$ to $\frac{3}{16}$. Subsequently, Anantharaman and Monk \cite{AN4} further refined it to $\frac{2}{9}$. Regarding the optimal value of $\alpha$ as a typical gap, it has been shown that $\alpha \leq \frac{1}{4}$ by Huber \cite{HUB2} and Cheng \cite{Cheng}. Wright \cite{WA} conjectured that $\alpha$ equals $\frac{1}{4}$. Very recently, Anantharaman and Monk \cite{AN7} claimed to have proved Wright's conjecture that $\alpha=\frac{1}{4}$. For the twisted Laplacian, we prove the following theorem concerning $\lambda_{1}^{\omega}(X)$:
\begin{theorem}\label{twigap}
  For a typical spectral gap $\alpha \in \left(0,\frac{1}{4}\right]$, any $0<\varepsilon<\alpha$ and $c>0$, there exists $g_{\alpha,c,\varepsilon}>0$ such that for any $g>g_{\alpha,c,\varepsilon}$, any $X \in \mathcal{A}_{g} \cap \mathcal{L}_{g} \cap \mathcal{Q}_{g}^{\alpha, \varepsilon/2}$ and any $\omega \in \mathcal{H}^{1}(X,\mathbb{R})$ with $||\omega||_{L^{2}}\leq c$, we have:
 \begin{equation}
 \Re \lambda_{1}^{\omega}(X)>\alpha-\varepsilon.
 \end{equation}
\end{theorem}
\begin{remark}
  Theorems \ref{uniweylfortwist}, \ref{uniweylfortwistsmalleigen}, and \ref{twigap} also hold for the harmonic form $\omega \in \mathcal{H}^{1}(X,\mathbb{R})$ satisfying the condition that for any constant $c$ and any $0<\epsilon<\frac{1}{8}$, the following inequality holds:
  $$||\omega||_{L^{2}}\leq cg^{\frac{1}{8}-\epsilon}, \quad \omega \in \mathcal{H}^{1}(X,\mathbb{R}).$$
  The fraction $\frac{1}{8}$ is derived from Theorem \ref{deh}. Additionally, under this condition, the constants $g_{\bullet}$ in Theorems \ref{uniweylfortwist}, \ref{uniweylfortwistsmalleigen}, and \ref{twigap} also depend on $\epsilon$.
\end{remark}
\subsection{Related works}
The twisted Laplacian \eqref{twistedlaplacian} is similar to the stationary damped wave equation (see e.g. Lebeau \cite{LB96}, Sj\"{o}strand \cite{SJ1}), which is also a non-self-adjoint perturbation of the Laplacian. For the stationary damped wave operator, the real and imaginary parts of the complex eigenvalues correspond to the frequencies of oscillation and the exponential decay rate of the damped wave, respectively. The spectral distribution of the damped wave operator on a fixed manifold has been well studied and here we only mention a few results and refer the reader to these papers for further references. Markus-Matsaev \cite{MAMV} (also Sj\"{o}strand \cite{SJ1}) first proved an analogue of Weyl law for the distribution of the real parts of the spectrum. For the distribution of the imaginary parts of the spectrum, Lebeau \cite{LB96} studied the spectral gap which is related to the decay rate of the energy of damped wave, and Sj\"{o}strand \cite{SJ1} obtained the average distribution. A recent monograph of Sjöstrand \cite{SJ2} systematically studied the spectral distribution of non-self-adjoint perturbations of the second-order elliptic operators.\\
\indent Heuristically, the spectral distribution of the Laplacian or its non-self-adjoint perturbations in the high frequency limit is related to the dynamics of the geodesic flow on the manifold. For hyperbolic surfaces or more general closed manifolds with negative curvatures, the geodesic flow is chaotic. The study of the spectral distribution of both the Laplacian and its non-self-adjoint perturbations in the high frequency limit belongs to the area of quantum chaos, see Nonnenmacher \cite{NS11} and Zelditch \cite{ZS19} for recent surveys. For damped wave operators on manifolds with chaotic geodesic flows, we refer to Schenck \cite{SE}, Jin \cite{JL}, and Dyatlov-Jin-Nonnenmacher \cite{DJN} for some recent results on better spectral gaps and Anantharaman \cite{AN2} for fractal Weyl upper bounds for the number of eigenvalues in certain strips.\\
\indent As the work of Monk \cite{ML1,ML2}, our results are about the spectral distribution in large scale limit instead of the high frequency limit. This large scale limit of hyperbolic surfaces can be viewed as another way to formulate quantum chaos; see Anantharaman's ICM talk \cite{AN3}.  Some recent works in this direction include Le Masson-Sahlsten \cite{LS,EL2} for quantum ergodicity, Gilmore-Le Masson-Sahlsten-Thomas \cite{GCLS} for $L^p$ estimate of eigenfunctions, Thomas \cite{TJ} for delocalization of eigenfunctions, and Rudnick \cite{RZ} for GOE statistics.\\
\indent The twisted Laplacian \eqref{twistedlaplacian} is also identified with the twisted Bochner Laplacian on the flat bundle constructed by the non-unitary representation of the fundamental group, see \cite{Kob}. M{\"u}ller \cite{MW}, Spilioti \cite{SP} and Frahm-Spilioti \cite{FS} studied the twisted trace formula and zeta functions, which relate the spectrum, the closed geodesics, and the representation. Furthermore, Naud-Spilioti \cite{NS} proved the Weyl law for the non-unitary twisted Bochner Laplacian with Teichm\"{u}ller type representation by the method of thermodynamic formalism. Their results are similar to those in \cite{AN2} and \cite{SJ1}, while the proof is very different.
\subsection*{Organization of the paper}
The paper is organized as follows: In Section \ref{sec2}, we prove Theorem \ref{infform}, which shows that the spectral distribution for $\Delta_{\omega}$ converges to the main term of Theorem \ref{uniweylfortwist} uniformly for a typical surface $X \in \mathcal{A}_{g}$. In Section \ref{sec3}, we prove Theorem \ref{semiweyltwi}, which estimates the remainder term in Theorem \ref{infform} by the $L^{\infty}$ norm of the harmonic form. In Section \ref{sec4}, we prove Theorem \ref{deh}, which controls the $L^{\infty}$ norm of the harmonic form by its $L^{2}$ norm on a typical surface $X\in \mathcal{L}_{g}$. In Section \ref{maintheoremsection}, we combine the results in Sections 3 and 4 to prove the main Theorems \ref{uniweylfortwist}, \ref{uniweylfortwistsmalleigen}, and \ref{twigap}. In Appendix \ref{app2}, we prove Theorem \ref{delocaleigen} which generalizes Theorem \ref{deh} to eigenforms.
\subsection*{Notation}
In the following, when the dependence of a constant on parameters is crucial, we indicate the parameters on which they depend. The notation $f \lesssim g$ or $f \gtrsim g$ means that there exists a universal constant $C>0$ such that $f \leq C g$ or $f\geq Cg$ for any choice of parameters, respectively. If the constant depends on some parameter $x$, we write $f\lesssim_{x} g$ or $f \gtrsim_{x} g$. $\sigma(T)$ denotes the set of the eigenvalues of an operator $T$. $d\mathrm{Vol}(x)$ and $dA(z)$ denote the volume elements of hyperbolic plane $\mathbb{H}$ and complex plane $\mathbb{C}$, respectively. $\Omega^{k}(X)$ denotes the space of $k$-forms on $X$.

\section{The convergence of spectral distribution}\label{sec2}
By \eqref{eq30}, we observe that the imaginary part of the eigenvalue, $\Im \lambda_{j}^{\omega}$, can be estimated by the $L^\infty$ norm of $\omega$. Therefore, we investigate the spectral distribution of the twisted Laplacian under the $L^\infty$ boundedness of $\omega$, rather than the $L^2$ boundedness of $\omega$ in main Theorem \ref{uniweylfortwist}. In this section, we establish the following theorem:
\begin{theorem}\label{infform}
  Fix a closed interval $[a,b]$ and a constant $c>0$. For any $X \in \mathcal{A}_{g}$ defined in \eqref{geoassump}, we have:
  \begin{equation}\label{inftyre}
    \lim _{g \rightarrow \infty} \sup _{X \in \mathcal{A}_g} \sup _{\substack{\omega \in \mathcal{H}^1(X, \mathbb{R}) \\\|\omega\|_{L^{\infty}} \leqslant c}} \left | \frac{N_{X}^{-\Delta_{\omega}}(\Re \lambda \in [a,b])}{\mathrm{Vol}(X)} - \frac{1}{4\pi}\int_{\frac{1}{4}}^{\infty}\mathbf{1}_{[a,b]}(\lambda)\tanh \left ( \pi\sqrt{\lambda-\frac{1}{4}} \right ) d\lambda \right |=0,
  \end{equation}
  and for any $\epsilon>0$, 
  \begin{equation}\label{inftyim}
    \lim _{g \rightarrow \infty} \sup _{X \in \mathcal{A}_g} \sup _{\substack{\omega \in \mathcal{H}^1(X, \mathbb{R}) \\\|\omega\|_{L^{\infty}} \leqslant c}} \frac{N_{X}^{-\Delta_{\omega}}(\Re \lambda \in [a,b], \Im \lambda \notin (-\epsilon,\epsilon))}{\mathrm{Vol}(X)}=0.
  \end{equation}
\end{theorem}
For any rectangle $R \subset \mathbb{C}$, Theorem \ref{infform} implies that: $$\lim_{g \to \infty}\frac{\mathrm{N}_{X_{g}}^{-\Delta_{\omega}}(\lambda \in R)}{\mathrm{Vol}\left(X_g\right)}=\frac{1}{4 \pi} \int_{\frac{1}{4}}^{\infty} \mathbf{1}_{R}(\lambda) \tanh \left(\pi \sqrt{\lambda-\frac{1}{4}}\right) \mathrm{d} \lambda,$$ where $X_{g}$ represents a uniformly discrete sequence of closed hyperbolic surfaces with injectivity radius $\mathrm{Inj}(X_{g})$ larger than some $r>0$, converging to $\mathbb{H}$ in the sense of Benjamini-Schramm, see \cite{LS} and \cite{ML2}.\\
\indent The proof of Theorem \ref{infform} follows Le Masson-Sahlsten \cite[Lemma 9.1]{LS}. It is based on the twisted Selberg pre-trace formula (e.g. see Anantharaman \cite[Proposition 7.1]{AN2}). Let the spectral parameter $r_{j}^{\omega}(X)$ be $\lambda^{\omega}_{j}(X)=\frac{1}{4}+r_{j}^{\omega}(X)^{2}$:
\begin{theorem}[Selberg pre-trace formula]\label{Sel}
  For $\omega \in \mathcal{H}^{1}(X,\mathbb{C})$:
  \begin{equation*}\label{Sel1}
  \sum_{j=0}^{\infty}h(r_{j}^{\omega}(X))=\frac{\mathrm{Vol}(X)}{4\pi}\int_{-\infty}^{\infty}h(r)r\tanh(\pi r)dr+\sum_{\gamma \in \Gamma-\{ \mathrm{Id} \}}\int_{D}e^{-\int_{\gamma}\omega}k(x,\gamma x)d\Vol(x),
  \end{equation*}
  where $D$ is a fundamental domain for the universal cover $\pi: \mathbb{H} \to X$, the radial kernel function $k(x,y)=K(d(x,y))$ is smooth enough, and decays faster than any exponential and $h=\mathcal{S}(K)$ is the Selberg transform of $K$, which is defined as the following:
  $$S(k)(r)=\int_{-\infty}^{\infty}e^{iru}g(u)du, \quad g(u)=\sqrt{2}\int_{|u|}^{\infty}\frac{K(\rho)\sinh \rho}{\sqrt{\cosh \rho -\cosh u}}d\rho.$$
\end{theorem}
The three terms in Theorem \ref{Sel} from left to right are called spectral average term, topological term, and geometric term, respectively. We choose $k$ to be the heat kernel $k(t,z,w)$:
\begin{equation*}
  \begin{aligned}
    k(t, z, w)=K(t,d(z,w))=\frac{\sqrt{2} e^{-\frac{t}{4}}}{(4 \pi t)^{\frac{3}{2}}} \int_{d(z, w)}^{\infty} \frac{u e^{-\frac{u^2}{4 t}} d u}{\sqrt{\cosh u-\cosh d(z, w)}}.
    \end{aligned}
\end{equation*} 
The Selberg transform of $K(t,\cdot)$ is 
$$h(t,r)=\mathcal{S}(K(t,\cdot))(r)=e^{-(\frac{1}{4}+r^{2})t}.$$
We obtain the following heat trace formula:
\begin{equation}\label{Selheat}
  \begin{aligned}
\sum_{j=0}^{\infty}e^{-t\lambda_{j}^{\omega}(X)}
=&\frac{\mathrm{Vol}(X)}{4\pi}\int_{-\infty}^{\infty}e^{-t(\frac{1}{4}+r^{2})}r\tanh(\pi r)dr\\
&+\sum_{\gamma \in \Gamma-\{\mathrm{Id}\}}\int_{D}e^{-\int_{\gamma}\omega}k(t,x,\gamma x)d\Vol(x).
  \end{aligned}
\end{equation}
In the following, we assume that $X \in \mathcal{A}_{g}$ and prove that the geometric term of \eqref{Selheat} is negligible when $g \to \infty$.
\subsection{Estimate of the heat trace}
In what follows, we fix a $t>0$ in heat trace formula \eqref{Selheat}. 
By \cite[Lemma 7.4.26]{BU}, there is a constant $C > 0$ such that for any $t > 0$,
\begin{equation*}
k(t,x,y)\leq Ct^{-1}e^{-\frac{d(x,y)^2}{8t}}.
\end{equation*}
Thus we have $e^{-\int_{\gamma}\omega}\leq e^{||\omega||_{L^{\infty}}\ell(\gamma)}$ and $k(t,x,\gamma x) \leq Ct^{-1}e^{-\frac{d(x,\gamma x)^2}{8t}}.$
Since $||\omega||_{L^{\infty}}\leq c$, we have 
\begin{equation*}
\sum_{\gamma\in \Gamma-\{ \mathrm{Id} \}}e^{-\int_{\gamma}\omega}\int_{D}k(t,x,\gamma x)d\Vol(x)\leq Ct^{-1}\sum_{\gamma\in \Gamma-\{ \mathrm{Id} \}}\int_{D}e^{cd(x,\gamma x)-\frac{d(x,\gamma x)^2}{8t}}d\Vol(x).
\end{equation*}
We need the following lemma for our estimate:
\begin{lemma}
    Let $\Gamma_{r}(x):=\{\gamma \in \Gamma: r< d(x,\gamma x)\leq r+1\}$ for any $x\in X$ then 
    \begin{equation}\label{latticecount}
      \# \Gamma_{r}(x) \leq \frac{\cosh(r+1+r_{g})-\cosh(r-r_{g})}{\cosh(r_{g})-1}\leq \frac{e^{r_{g}+1}}{r_{g}^{2}}e^{r}.
    \end{equation}
\end{lemma}
\begin{proof}
    For any $x \in \mathbb{H}$, since $\mathrm{Inj}_{x}(X)=\frac{1}{2}\min\limits_{\gamma \in \Gamma-\{ \mathrm{Id} \}}d(x,\gamma x)$, for any $X \in \mathcal{A}_{g}$ defined in \eqref{geoassump}, we find that for any $\gamma \in \Gamma-\{ \mathrm{Id} \},$ $$d(x,\gamma x) \geq 2\mathrm{Inj}(X) \geq 2r_{g}.$$
    It implies that:
    $$B(x,r_{g})\cap B(\gamma x, r_{g})=\emptyset, \ \forall \gamma \in \Gamma-\{ \mathrm{Id} \},$$
    and for any $y \in B(\gamma x, r_{g})$, $\gamma \in \Gamma_{r}(x)$,
    \begin{equation*}
      d(y,x)\leq d(y,\gamma x)+d(x,\gamma x)< r_{g}+r+1, \quad d(y,x)\geq d(x,\gamma x)-d(y,\gamma x)> r-r_{g}.
    \end{equation*}
    Then we find that
    \begin{equation*}
      \bigsqcup_{\gamma \in \Gamma_{r}}B(x,r_{g})\subset B(x,r+1+r_{g})\setminus B(x,r-r_{g}).
    \end{equation*}
    The area of $B(x,r)$ on the hyperbolic plane $\mathbb{H}$ is $2\pi(\cosh(r)-1)$, so we obtain that:
    $$
      \# \Gamma_{r}(x)\cdot 2\pi(\cosh(r_{g})-1) \leq 2\pi (\cosh(r+1+r_{g})-\cosh(r-r_{g})).
    $$
    It implies that
    \begin{equation*}
      \# \Gamma_{r}(x) \leq \frac{\cosh(r+1+r_{g})-\cosh(r-r_{g})}{\cosh(r_{g})-1}\leq \frac{e^{r_{g}+1}}{r_{g}^{2}}e^{r}.
    \end{equation*}
\end{proof}
Now we can divide $\Gamma$ into $\bigsqcup_{r=[2\mathrm{Inj}_{x}(X)]}^{\infty}\Gamma_{r}(x)$. Since $d(x,\gamma x)\geq 2\mathrm{Inj}_{x}(X)\geq 2r_{g}>0,$ we use \eqref{latticecount} to obtain that:
\begin{equation*}
  \begin{aligned}
  &\sum_{\gamma\in \Gamma-\{ \mathrm{Id} \}}e^{cd(x,\gamma x)-\frac{d(x,\gamma x)^2}{8t}} \leq \sum_{r=[2\mathrm{Inj}_{x}(X)]}^{\infty}\sum_{\gamma \in \Gamma_{r}(x)}e^{cd(x,\gamma x)-\frac{d(x,\gamma x)^2}{8t}}\\
  \leq &\sum_{r=[2\mathrm{Inj}_{x}(X)]}^{\infty}\frac{e^{r_{g}+1}}{r_{g}^{2}}e^{r}\cdot e^{cr-\frac{r^{2}}{8t}} \leq \frac{e^{r_{g}+1}}{r_{g}^{2}}e^{2t(c+1)^{2}}\sum_{r=[2\mathrm{Inj}_{x}(X)]}^{\infty}e^{-\frac{(r-4t(c+1))^{2}}{8t}}.
  \end{aligned}  
\end{equation*}
We take $R=\frac{1}{6}\log g>2t(c+1)$ and divide the domain $D$ of the integral in the geometric term in \eqref{Selheat} into $D(R):=\{x\in D: \mathrm{Inj}_{x}(X)<R\}$ and $D(R)^{\complement}$:
\begin{equation*}
  \frac{1}{\mathrm{Vol}(D)}\int_{D(R)}\sum_{\gamma\in \Gamma-\{ \mathrm{Id} \}}e^{-\int_{\gamma}\omega}k(t,x,\gamma x)d\Vol(x)\lesssim_{t,c} \frac{e^{r_{g}+1}}{r_{g}^{2}}\frac{\mathrm{Vol}(D(R))}{\mathrm{Vol}(D)},
\end{equation*}
and 
\begin{equation*}
  \frac{1}{\mathrm{Vol}(D)}\int_{D(R)^{\complement}}\sum_{\gamma\in \Gamma-\{ \mathrm{Id} \}}e^{-\int_{\gamma}\omega}k(t,x,\gamma x)d\Vol(x)\lesssim_{t,c} \frac{e^{r_{g}+1}}{r_{g}^{2}}e^{-\frac{(R-2t(c+1))^{2}}{2t}}.
\end{equation*}
Then we deduce that:
\begin{equation*}
  \begin{aligned}
  \frac{1}{\mathrm{Vol}(D)}\int_{D}\sum_{\gamma\in \Gamma-\{ \mathrm{Id} \}}e^{-\int_{\gamma}\omega}k(t,x,\gamma x)d\Vol(x) & \lesssim_{t,c}  \frac{1}{r_{g}^{2}} \left (\frac{\mathrm{Vol}(D(R))}{\mathrm{Vol}(D)}+e^{-\frac{(R-2t(c+1))^{2}}{2t}} \right )  \\
  & \lesssim_{t,c}  g^{-\frac{1}{4}}(\log g)^{-\frac{9}{8}} .
  \end{aligned}
\end{equation*}
By \eqref{Selheat}, for any $\log g> 12(c+1)t,$ 
\begin{equation}\label{heatker}
  \sup_{X \in \mathcal{A}_{g}} \left |\frac{1}{\mathrm{Vol}(X)}\sum_{j=0}^{\infty}e^{-t\lambda_{j}^{\omega}(X)}-\frac{1}{4\pi}\int_{\mathbb{R}}e^{-t(\frac{1}{4}+\rho^2)}\tanh(\pi\rho)\rho d\rho \right | \lesssim_{t,c}  g^{-\frac{1}{4}}(\log g)^{-\frac{9}{8}}.
\end{equation}
\subsection{Asymptotic distribution of the spectrum}
By taking the inner product with $\phi_{j}(x)$ in the $L^2(X)$ space on both sides of \eqref{equation}, we obtain the following:
\begin{equation*}
  (\Delta \phi_{j},\phi_{j})_{L^2}-2(\langle\omega,d\phi_{j}\rangle_{x},\phi_{j}(x))_{L^2}+(|\omega|_{x}^2\phi_{j},\phi_{j})_{L^2}+\Re \lambda^{\omega}_{j}(X) +i \Im \lambda^{\omega}_{j}(X)=0.
\end{equation*}
Since $\omega \in \mathcal{H}^1(X,\mathbb{R})$, then\footnote{Here we use the property of harmonic 1-form $\omega$: $(\omega,df)_{L^2}=(\delta \omega, f)_{L^2}=0$ for all $f\in C^{\infty}(X)$, where $\delta$ is the dual operator of $d$.}
\begin{equation*}
  \int_{X}\langle\omega,d\phi_{j}\rangle_{x}\overline{\phi_{j}(x)}d\Vol(x)+\overline{\int_{X}\langle\omega,d\phi_{j}\rangle_{x}\overline{\phi_{j}(x)}d\Vol(x)}=\int_{X}\langle\omega,d\left(|\phi_{j}|^2\right)\rangle_{x}d\Vol(x)=0.
\end{equation*}
It implies $(\langle\omega,d\phi_{j}\rangle_{x},\phi_{j}(x))_{L^2} \in i\mathbb{R}$. Then by comparing the real and imaginary parts on both sides, we have:
\begin{equation}\label{eq30}
  \begin{gathered}
    \Re \lambda^{\omega}_{j}(X)=\int_{X}(|d\phi_{j}|_{x}^2-|\omega|_{x}^2|\phi_{j}(x)|^2)d\Vol(x), \\
    \Im \lambda^{\omega}_{j}(X)=-2i\int_{X}\langle \omega,d\phi_{j} \rangle_{x}\overline{\phi_{j}(x)}d\Vol(x).\\
  \end{gathered}
\end{equation}
Here $||\phi_{j}||_{L^{2}(X)}=1$. By $||\omega||_{L^{\infty}}\leq c$, if $\Re \lambda^{\omega}_{j}(X) \in [a,b]$, then 
\begin{equation*}
|\Im \lambda^{\omega}_{j}(X)|\leq 2c\int_{X} |d\phi_{j}|_{x}|\phi_{j}(x)| d\Vol(x) \leq 2c \left ( \int_{X} |d\phi_{j}|_{x}^2d\Vol(x) \right )^{\frac{1}{2}}\leq 2c\sqrt{b+c^{2}}.
\end{equation*}
Take $$\Omega :=\left\{x+yi: x\geq a-1, \ |y|\leq 2c\sqrt{b+c^{2}}+1 \right\},$$ and define 
\begin{equation*}
C_{\infty}(\Omega):=\{f\in C(\Omega):\lim_{\Re z \to \infty}f(z)=0\}.
\end{equation*}
The Weierstrass approximation theorem implies that the closure of the algebra generated by $\{e^{-tz},z\in \Omega:\forall t>0\}$ is $C_\infty(\Omega)$. In other words, for any $f\in C_\infty(\Omega)$ and $\epsilon>0$, there exist $a_k\in\mathbb{C}, t_k>0$ , $k=1,\ldots,n$ such that $$\left | \left | f(z)-\sum_{k=1}^{n}a_{k}e^{-t_{k}z} \right | \right |_{L^{\infty}(\Omega)}<\epsilon.$$
We denote the space of compactly supported continuous functions on $\Omega$ by $C_{0}(\Omega)$,
\begin{theorem}\label{cptsupp}
    For any $f\in C_{0}(\Omega)\subset C_{\infty}(\Omega)$, we have 
    $$\lim_{g \to \infty}\sup_{X \in \mathcal{A}_{g}}\left |\frac{1}{\mathrm{Vol}(X)}\sum_{j=0}^{\infty} f(\lambda_{j}^{\omega}(X))-\frac{1}{4\pi}\int_{\mathbb{R}}f \left ( \frac{1}{4}+\rho^{2} \right ) \tanh(\pi \rho)\rho d\rho \right |=0.$$
\end{theorem}
\begin{proof}
    The method is the same as \cite[Theorem 9]{LS}. Let $h(z)=f(z)e^{z}\in C_{0}(\Omega)$ then for any $\epsilon>0$ we can choose $a_{k}\in \mathbb{C}, \ t_{k}>0, \ k=1,\cdots,m$ and such that
    \begin{equation}\label{appro}
    ||h(z)-h_{\epsilon}(z)||_{L^{\infty}(\Omega)}<\epsilon, \quad h_{\epsilon}(z)=\sum_{k=1}^{m}a_{k}e^{-t_{k}z},
    \end{equation}
    and define 
    \begin{equation*}
    S_{X}(f)=\frac{1}{\mathrm{Vol}(X)}\sum_{j=0}^{\infty}f(\lambda^{\omega}_{j}(X)), \quad
    I(f):=\frac{1}{4\pi}\int_{\mathbb{R}}f \left ( \frac{1}{4}+\rho^{2} \right ) \tanh(\pi \rho)\rho d\rho.
    \end{equation*}
    Thus we have
    \begin{equation*} 
    |S_{X}(f)-I(f)|\leq |S_{X}(f)-S_{X}(h_{\epsilon}e^{-z})|+|S_{X}(h_{\epsilon}e^{-z})-I(h_{\epsilon}e^{-z})|+|I(h_{\epsilon}e^{-z})-I(f)|.
    \end{equation*}
    By \eqref{appro}, the first term 
    \begin{equation*}
    |S_{X}(f)-S_{X}(h_{\epsilon}e^{-z})|\leq ||h-h_{\epsilon}||_{L^{\infty}(\Omega)}S_{X}(e^{-z})<\epsilon S_{X}(e^{-z}).
    \end{equation*}
    By \eqref{heatker}, if $ X \in \mathcal{A}_{g}$ and $\log g>12(c+1) \sum_{i=1}^{m}t_{i},$ then the second term
    \begin{equation*}
      \begin{aligned}
    |S_{X}(h_{\epsilon}e^{-z})-I(h_{\epsilon}e^{-z})| &  \leq  \sum_{k=1}^{m}|a_{k}|\cdot|S_{X}(e^{-(t_{k}+1)z})-I(e^{-(t_{k}+1)z})| \\ 
     & \lesssim_{t_1,\cdots,t_{m},c}\left ( \sum_{k=1}^{m}|a_{k}| \right ) g^{-\frac{1}{4}}(\log g)^{-\frac{9}{8}}.
      \end{aligned}
    \end{equation*}
     Now we take $g \to \infty$, by \eqref{heatker}:
    \begin{equation*}
    \limsup_{g \to \infty}\sup_{X \in \mathcal{A}_{g}}|S_{X}(f)-I(f)| \leq \epsilon I(e^{-z})+|I(h_{\epsilon}e^{-z})-I(f)|.
    \end{equation*}
    Let $\epsilon \to 0$, and then we finish the proof by dominated convergence theorem the right-hand side tends to $0$. 
  \end{proof}
    Finally, we give the proof of Theorem \ref{infform}, we define $f(\lambda)=\mathbf{1}_{\Omega_{a,b}}(\lambda),$ where $\Omega_{a,b}$ is defined by $\{\lambda \in \Omega: \Re \lambda \in [a,b]\}.$
    It follows that $$S_{X}(f)=\frac{N_{X}^{-\Delta_{\omega}}(\Re \lambda \in [a,b])}{\mathrm{Vol}(X)}, \quad I(f)=\frac{1}{4\pi}\int_{\frac{1}{4}}^{\infty}\mathbf{1}_{[a,b]}(\lambda)\tanh \left ( \pi\sqrt{\lambda-\frac{1}{4}} \right )d\lambda.$$
    Now we construct two sequence $\underline{f_{n}}, \overline{f_{n}} \in C_{0}^{\infty}(\Omega)$ such that $0\leq \underline{f_{n}}\leq f \leq \overline{f_{n}} \leq 1$ and $$\lim_{n \to \infty}\underline{f_{n}}=\lim_{n \to \infty}\overline{f_{n}}=f.$$
    Then we find 
    $$S_{X}(\underline{f_{n}})\leq S_{X}(f)\leq S_{X}(\overline{f_{n}})$$
    Now we consider 
    $$S_{X}(f)-I(f)\leq S_{X}(\overline{f_{n}})-I(\overline{f_{n}})+I(\overline{f_{n}})-I(f)\leq \left | S_{X}(\overline{f_{n}})-I(\overline{f_{n}}) \right |+\left | I(\overline{f_{n}})-I(f) \right |,$$
    and 
    $$S_{X}(f)-I(f)\geq S_{X}(\underline{f_{n}})-I(\underline{f_{n}})+I(\underline{f_{n}})-I(f)\geq - \left | S_{X}(\underline{f_{n}})-I(\underline{f_{n}}) \right |-\left | I(\underline{f_{n}})-I(f) \right |.$$
    It implies that 
    $$\left | S_{X}(f)-I(f) \right | \leq  \left | S_{X}(\overline{f_{n}})-I(\overline{f_{n}}) \right |+\left | S_{X}(\underline{f_{n}})-I(\underline{f_{n}}) \right |+ \left | I(\overline{f_{n}})-I(f) \right |+\left | I(\underline{f_{n}})-I(f) \right |.$$
    Then we take the supremum of $X\in \mathcal{A}_{g}$, let $g\to \infty$ then $n \to \infty,$ we obtain \eqref{inftyre} by Theorem \ref{cptsupp}.\\
    \indent To prove \eqref{inftyim}, we consider $\Omega_{a, b, \epsilon}$ which is defined by $\{\lambda \in \Omega: \Re \lambda \in [a,b], \ |\Im \lambda| \geq \epsilon\}$ and $f(\lambda)=\mathbf{1}_{\Omega_{a, b, \epsilon}}(\lambda),$ then $$S_{X}(f)=\frac{N_{X}^{-\Delta_{\omega}}(\Re \lambda \in I, \Im \lambda \notin (-\epsilon,\epsilon))}{\mathrm{Vol}(X)}, \quad I(f)=0.$$
    The same argument above gives \eqref{inftyim}.
    \section{The uniform estimate of the remainder term}\label{sec3}
We obtain the main term of the Weyl law of twisted Laplacian in Theorem \ref{infform}. Therefore, we investigate the remainder term of \eqref{inftyre}. As in \cite{SJ1, SJ2, AN2}, we reformulate the spectral problem by semiclassical microlocal analysis, see \cite{DS, ZM, DZ19}.\\
 \indent We introduce the semiclassical parameter $h$ with $0<h\leq 1$. For an $\omega \in \mathcal{H}^{1}(X,\mathbb{R})$, let $P=-h^{2}\Delta$, $P_{\omega}=-h^{2}\Delta_{\omega}$ and $Q_{\omega}=\langle -2ih\omega, d(\cdot) \rangle$, then 
\begin{equation*}
P_{\omega}=P+ihQ_{\omega}-h^{2}|\omega|_{x}^{2}.
\end{equation*}
In this section, we assume that the upper bound of $||\omega||_{L^\infty}$ is a bounded function $c(g)$ of the genus $g$, then we prove the following theorem:
    \begin{theorem}\label{semiweyltwi}
      There exists a universal constant $C_{0}^{\prime}>0$, such that for any $c>0$, any function $c(g)$ satisfying $0 \leq c(g) \leq c$ for all $g \in \mathbb{N}^{+}$, any $g, h$ satisfying
      \begin{equation*}\label{eq35}
       0<h \leq 1, \quad c(g)h<10^{-4}, \quad g>g_{0}, 
      \end{equation*}
      where $g_0$ is defined in Theorem \ref{uniweyl}, any interval $I=[a,b]$ satisfying
      \begin{equation}\label{eq33}
      -2\leq a \leq b \leq 1, \quad b\geq 0, \quad b-a\geq 8000c(g)h,
      \end{equation}
      and any $X \in \mathcal{A}_{g}$, we have:
      \begin{equation}
        \sup _{\substack{\omega \in \mathcal{H}^1(X, \mathbb{R}) \\\|\omega\|_{L^{\infty}} \leqslant c(g)}}\left | N^{P_{\omega}}_{X}(\Re \lambda \in I)-N_{X}^{P}(\lambda \in I) \right | \leq C_{0}^{\prime}\mathrm{Vol}(X)\left ( \sqrt{\frac{1}{\log g}}+c(g) \right)h^{-1}.
      \end{equation}
    \end{theorem} 
    When $c(g)=o(1)$ as $g \to \infty$, we obtain the remainder term of Weyl law of twisted Laplacian for $X\in \mathcal{A}_g$ by combining Theorem \ref{semiweyltwi} with Theorem \ref{uniweyl}. \\
    \indent For any $c>0$, any $c(g)$ satisfying $0 \leq c(g) \leq c$ for all $g \in \mathbb{N}^{+}$, and any interval $I=[a,b]$ which satisfies \eqref{eq33}, let $\delta=4000c(g)h$, ensuring that $b-a \geq 2\delta$. Then we define
    \begin{equation}
    I_{\delta}:=\partial I+(-\delta,\delta)=(a-\delta,a+\delta)\cup (b-\delta,b+\delta),
    \end{equation}
    which represents the $\delta$-neighborhood of $\partial I=\{a,b\}$. For any $X\in \mathcal{M}_g$, we denote $N^{P}_{X}(\lambda \in I_{\delta})$ as the count of eigenvalues for $P$ within the $I_{\delta}$.\\
    \indent In fact, Theorem \ref{semiweyltwi} is a corollary of the following theorem, which shows that the difference between the number of eigenvalues for Laplacian and twisted Laplacian is controlled by the count of eigenvalues for $P$ within the $I_{\delta}$, $\delta$-neighborhood of $\partial I$.
    \begin{theorem}\label{perturoflaplacian}
    There exists a universal constant $C_{0}^{\prime\prime}>0$, such that for any $c>0$, any function $c(g)$ satisfying $0 \leq c(g) \leq c$ for all $g \in \mathbb{N}^{+}$, any $g, h$ satisfying 
    \begin{equation}\label{eq46}
    0<h\leq 1, \ c(g)h<10^{-4},
    \end{equation}
    any interval $I=[a,b]$ which satisfies \eqref{eq33}, and any $X \in \mathcal{M}_g$, we have:\footnote{We note that this theorem does not need $X\in \mathcal{A}_{g}$, we use $X \in \mathcal{A}_{g}$ to estimate $N^{P}_{X}(\lambda \in I_{\delta})$.}
    \begin{equation}
      \sup _{\substack{\omega \in \mathcal{H}^1(X, \mathbb{R}) \\\|\omega\|_{L^{\infty}} \leqslant c(g)}}\left | N^{P_{\omega}}_{X}(\Re \lambda \in I)-N_{X}^{P}(\lambda \in I) \right |\leq C_{0}^{\prime\prime} N^{P}_{X}(\lambda \in I_{\delta}),
    \end{equation}
    where $I_{\delta}$ is defined as above.
    \end{theorem}
    \begin{remark}
      For the proof of the main Theorem \ref{uniweylfortwist}, we will choose $c(g)=cg^{-\frac{1}{48}}$ by Theorem \ref{deh}. Our proof of Theorem \ref{perturoflaplacian} follows the proof of Sjöstrand \cite[Theorem 14.1.1]{SJ2}, but we must be careful about the dependence on the constants.
    \end{remark}
    In the following, we will prove Theorem \ref{perturoflaplacian} from the subsection \ref{SS1} to \ref{SS4} and prove Theorem \ref{semiweyltwi} in the subsection \ref{SS5}. In what follows, we choose an arbitary $c>0$ and an arbitary function $c(g)$ with $0 \leq c(g) \leq c$ for all $g \in \mathbb{N}^{+}$ and always assume that $g,h$ satisfy \eqref{eq46} and the interval $I=[a,b]$ satisfies \eqref{eq33}.
    \subsection{Perturbation of the operator}\label{SS1}
    In this subsection, we construct several perturbations of $P$. We define $\delta$-perturbation $P^{\delta}$ as the \cite[(14.1.6)]{SJ1}, where $\delta=4000c(g)h$. Let $\{\psi_{i}\}_{i=0}^{\infty}$ be the orthogonal normal basis of $L^{2}(X)$ such that $$P\psi_{i}=\lambda_{i}\psi_{i}, \ \text{with} \ 0=\lambda_{0}<\lambda_{1}\leq \lambda_{2}\leq \cdots.$$
    We define
    \begin{equation}\label{perop}
    \tilde{\lambda}_{j}=\left\{\begin{array}{l}
      a-\delta \text { when } \lambda_{j} \in (a-\delta,a) \\
      a+\delta \text { when } \lambda_{j} \in [a,a+\delta) \\
      b-\delta \text { when } \lambda_{j} \in (b-\delta,b] \\
      b+\delta \text { when } \lambda_{j} \in (b,b+\delta) \\
      \lambda_{j} \text { when } \lambda_{j} \in \mathbb{R} \backslash I_{\delta}
      \end{array}\right.
    \end{equation}
    and 
    \begin{equation*}
      P^{\delta}u=\sum_{j=0}^{\infty} \widetilde{\lambda}_{j} ( u,\psi_{j} )_{L^{2}} \psi_{j}.
    \end{equation*}
    By the construction of $P^{\delta}$, the operator norm $||P^{\delta}-P||\leq \delta$. Thus we have: 
      \begin{equation} \label{eq1}
        ||P^{\delta}-P||_{\mathrm{tr}}\leq N^{P}_{X}(\lambda \in I_{\delta}) \delta,
      \end{equation}
      and 
      \begin{equation} \label{eq2}
        N^{P^{\delta}}_{X}(\lambda \in I)=N^{P}_{X}(\lambda \in I).
      \end{equation}
    For any $\omega \in \mathcal{H}^{1}(X,\mathbb{R})$ with $||\omega||_{L^{\infty}}\leq c(g)$, we introduce
    \begin{equation*}
    P^{\delta}_{\omega}=P^{\delta}+ihQ_{\omega}-h^{2}|\omega|_{x}^{2}.
    \end{equation*}
    If some $u \in C^{\infty}(X)$ satisfies
    $$P^{\delta}_{\omega}u=z u, \quad ||u||_{L^{2}(X)}=1,$$
    similar to \eqref{eq30}, we have:
    $$\Re z =(P^{\delta}u,u)_{L^2}-h^{2}(|\omega|_{x}^{2}u,u)_{L^2}, \quad \Im z=h(Q_{\omega}u,u)_{L^2}.$$
    So we obtain 
         $$\Re z\geq (Pu,u)_{L^2}+((P^{\delta}-P)u,u)_{L^2}-c(g)^{2}h^{2} \geq \int_{X}|h\nabla u|^{2}d\Vol(x) -\delta -c(g)^{2}h^{2},$$
         $$|\Im z|\leq 2h||\omega||_{L^{\infty}} \left ( \int_{X}|h\nabla u|^{2}d\Vol(x) \right )^{\frac{1}{2}} \leq 2c(g)h\sqrt{\Re z+\delta+c(g)^{2}h^{2}}.$$
    For $c(g)h<10^{-4}$, then $\delta+c(g)^{2}h^{2}<1$, 
    \begin{equation}\label{imofspec}
    z \in \sigma(P^{\delta}_{\omega}), \ \Re z \in I \implies |\Im z| \leq 4c(g)h.
    \end{equation}
    We find the region that contains no eigenvalue of $P^{\delta}_{\omega}$ in the following lemma:
    \begin{lemma}
      We take $\epsilon=20c(g)h=\frac{\delta}{200},$ then for any $d(z,\sigma(P^{\delta}))>\epsilon$, $\Re z \in [-3,3]$ and $|\Im z| \leq 2$,
      \begin{equation}\label{eq3}
        \left\|(z-P_{\omega}^{\delta})^{-1}\right\|< 2\epsilon^{-1}.
      \end{equation}
      As corollary,
      $$\sigma(P^{\delta}_{\omega})\cap \{ z: \Re z \in \partial I+(-\delta+\epsilon, \delta-\epsilon) \}=\emptyset.$$
    \end{lemma}
    \begin{proof}
    We notice that 
    $$(z-P^{\delta}_{\omega})(z-P^{\delta})^{-1}=\Id-K,$$
    where $$K=ihQ_{\omega}(z-P^{\delta})^{-1}-h^{2}|\omega|_{x}^{2}(z-P^{\delta})^{-1}.$$
    Here we write
    $$Q_{\omega}(z-P^{\delta})^{-1}=Q_{\omega}(-\Id-P)^{-1}(-\Id-P)(z-P^{\delta})^{-1},$$
    where
    \begin{equation*}
      \begin{aligned}
    ||(-\Id-P)(z-P^{\delta})^{-1}|| & =||(z-P^{\delta}-(1+z)+P^{\delta}-P)(z-P^{\delta})^{-1}|| \\ 
                                  & <1+|z+1|\epsilon^{-1}+\delta\epsilon^{-1}.
      \end{aligned}
    \end{equation*}
    Now we estimate $Q_{\omega}(\Id+P)^{-1},$ for any smooth $f$, $$||Q_{\omega}f||_{L^{2}(X)}^{2}\leq 4c(g)^{2}\int_{X}|h\nabla f|^{2}d\Vol(x)\leq 4c(g)^{2}||(\Id+P)^{\frac{1}{2}}f||^{2}_{L^{2}(X)}.$$
    Then for any $u\in C^{\infty}(X),$
    \begin{equation*}
      \begin{aligned}
    ||Q_{\omega}(\Id+P)^{-1}u||_{L^{2}(X)} & \leq 2c(g)||(\Id+P)^{-\frac{1}{2}}u||_{L^{2}(X)}=2c(g) \left ( \sum_{i=0}^{\infty}\frac{(u,\psi_{i})^{2}}{1+\lambda_{i}} \right )^{\frac{1}{2}}\\
    & \leq 2c(g)||u||_{L^{2}(X)}.
      \end{aligned}
    \end{equation*}
    So we prove that 
    \begin{equation}\label{eq11}
    ||Q_{\omega}(\Id+P)^{-1}||\leq 2c(g).
    \end{equation}
    Then $$||K||\leq h\cdot 2c(g) \cdot (1+|z+1|\epsilon^{-1}+\delta\epsilon^{-1})+c(g)^{2}h^{2}\epsilon^{-1}.$$
    For $c(g)h<10^{-4}$, since $|z+1|\leq \sqrt{4^{2}+2^{2}}=3\sqrt{2},$ for $z\in [-3,3]+i[-2,2]$,
    $$||K||\leq \frac{201}{8000}+\frac{3\sqrt{2}}{10}+\frac{1}{20} \cdot \frac{1}{8000}<\frac{1}{2}.$$
    By Neumann series, $\Id-K$ is invertible, so $z-P^{\delta}_{\omega}$ is also invertible,
    $$(z-P_{\omega}^{\delta})^{-1}=(z-P^{\delta})^{-1}(\Id-K)^{-1},$$
    then
    $$
      ||(z-P_{\omega}^{\delta})^{-1}||\leq \frac{||(z-P^{\delta})^{-1}||}{1-||K||}< 2\epsilon^{-1}.
    $$
  \end{proof}
  By a deformation argument: consider $P^{\delta}_{t\omega}$ for $t\in [0,1]$, let $\Omega:=I+i\left [-\frac{b-a}{2},\frac{b-a}{2} \right ]$, where $\frac{b-a}{2}\geq \delta=200\epsilon,$ then for all $t\in[0,1]$, $$\partial \Omega \subset (\sigma(P^{\delta})+\overline{D(0,\epsilon)})^{\complement} \subset \sigma(P^{\delta}_{t\omega})^{\complement}.$$
  We note $$N^{P^{\delta}_{t\omega}}_{X}(\lambda \in \Omega)=\mathrm{Tr}\frac{1}{2\pi i}\int_{\partial\Omega}(z-P^{\delta}_{t\omega})^{-1}dz,$$
  The right side is continuous for $t$, thus,
      \begin{equation}\label{eq4}
        N^{P^{\delta}_{\omega}}_{X}(\Re \lambda \in I)=N^{P_{\delta}}_{X}(\lambda \in I)=N^{P}_{X}(\lambda \in I).
      \end{equation}
    \subsection{Estimate the number of zeros of holomorphic functions}\label{SS2}
    In this subsection, we introduce a crucial theorem \cite[Theorem 12.1.2]{SJ2} about counting the zeros of holomorphic functions. Let $\Omega \Subset \mathbb{C}$ be an open set and $\gamma=\partial \Omega$. Let $r \colon \gamma \mapsto (0,\infty)$ be a Lipschitz function of Lipschitz modulus no more than $\frac{1}{2}$:
    \begin{equation}\label{eq25}
    |r(x)-r(y)|\leq \frac{1}{2}|x-y|, \ \forall x,y \in \gamma.
    \end{equation}
    We further assume that $\gamma$ is Lipschitz in the following precise sense, we say that $r$ is a Lipschitz weight for $\gamma$ if There exists a constant $C_1$ such that for every $x \in \gamma$ there exist new affine coordinates $\widetilde{y}=(\widetilde{y}_{1},\widetilde{y}_{2})$ of the form $\widetilde{y}=U(y-x)$, $y \in \mathbb{C}\cong \mathbb{R}^2$ being the old coordinates, where $U=U_{x}$ is orthogonal, such that the intersection of $\Omega$ and the rectangle 
    $$R_{x}:=\{y \in \mathbb{C}: |\widetilde{y}_{1}|<r(x), \ |\widetilde{y}_{2}|<C_{1}r(x)\}$$
    takes the form 
    \begin{equation}\label{eq12}
      \{y \in R_{x}: \widetilde{y}_{2}>f_{x}(\widetilde{y}_{1}), \ |\widetilde{y}_{1}|<r(x)\},
    \end{equation}
    where $f_{x}(\widetilde{y}_{1} )$ is Lipschitz on $[-r(x),r(x)]$, and $|f_{x}(t)-f_{x}(s)|\leq C_{1}|t-s|$. Notice that our assumption \eqref{eq12} remains valid if we decrease $r$.
    \begin{theorem}[Sjöstrand, \cite{SJ2}]\label{holodistri}
    Let $\Omega \Subset \mathbb{C}$ be an open simply connected set and have Lipschitz boundary $\gamma=\partial \Omega$ with an associated Lipschitz weight $r$ satisfies our previous assumption \eqref{eq25} and \eqref{eq12}. Let $z_{j}^{0} \in \gamma, j \in \mathbb{Z} / N \mathbb{Z}$ be distributed along the boundary in the positively oriented sense such that 
    \begin{equation}
    \frac{r(z_{j}^{0})}{4} \leq |z^{0}_{j+1}-z^{0}_{j}|\leq \frac{r(z_{j}^{0})}{2}.
    \end{equation}
    Then there exists a positive constant $C_{1}^{0}$, depending only on the constant $C_{1}$ in our assumption, such that every $C_{2}\geq C_{1}^{0}$, there exists a constant $C_{3}>0$ such that we have the following: \\
    Let $0<h \leq 1$ and let $\phi$ be a continuous function defined on some neighborhood of the closure of $W:=\bigcup_{z \in \gamma}D(z,r(z))$ such that $$\mu:= \Delta \phi(z)dA(z)$$ is a finite Radon measure there and denote by the same symbol a distribution extension to $\Omega \cup W$. Then there exist 
    $\widetilde{z}_{j} \in D\left (z_{j}^{0},\frac{r(z_{j}^{0})}{2C_{2}}\right )$ such that if $u$ is a holomorphic function on $\Omega \cup W$ satisfying:
    \begin{equation}
    h \log |u(z)|\leq \phi(z) , \quad \forall z \in W,
    \end{equation}
    \begin{equation}
    h \log |u(\widetilde{z}_{j})| \geq \phi(\widetilde{z}_{j})-\epsilon_{j}, \quad \forall j\in \mathbb{Z} / N \mathbb{Z}.
    \end{equation}
    Then 
    \begin{equation}
    \left |\#(u^{-1}(0)\cap \Omega)-\frac{1}{2\pi h}\mu(\Omega) \right |\leq \frac{C_{3}}{h} \left ( |\mu|(W)+\sum_{j \in \mathbb{Z} / N \mathbb{Z}}\epsilon_{j} \right ).
    \end{equation}
    \end{theorem}
    \indent In our situation, $\Omega=[a,b]+i\left [-\frac{b-a}{2},\frac{b-a}{2} \right ]\subset [-2,1]+i\left [-\frac{3}{2},\frac{3}{2} \right ] \Subset [-3,3]+i[-2,2]$, $\Omega$ depends on $g$ and $h$. $\gamma =\partial \Omega$ is Lipschitz boundary with an associated Lipschitz weight $$\widetilde{r}(z)=4\epsilon+\frac{|\Im z|}{4}, \quad  z \in \mathbb{C}$$
    of Lipschitz modulus $\frac{1}{4}$. Since $\Omega$ is a square, then we can make $f_{x}$ be a Lipschitz function with modulus $C_{1}=1$ in \eqref{eq12} for each $x\in \gamma$. In the following, we apply Theorem \ref{holodistri} to determinants of a Fredholm operator whose zeros correspond to the spectrum of $P_{\omega}$ in $\Omega$. 
\subsection{Zero of determinants}\label{SS3}
    Recall the \eqref{imofspec}, we find:
    \begin{equation}\label{eq50}
    |\Im z | \leq 4c(g)h=\frac{\epsilon}{5}<\frac{\epsilon}{4}, \ \text{for any} \ z \in \sigma(P^{\delta}_{\omega})\cap \Omega.
    \end{equation}
    Now we consider the neighborhood of $\partial \Omega$, see the above Figure 1:
    $$\widetilde{W}=\bigcup_{z \in \partial \Omega}D(z,\widetilde{r}(z)) \quad \text{and} \quad W=\bigcup_{z \in \partial \Omega}D \left ( z,r(z) \right ),$$
    where $r(z)=\frac{\widetilde{r}(z)}{2}.$
    \begin{figure}\label{region}
      \center
  \includegraphics[scale=0.4]{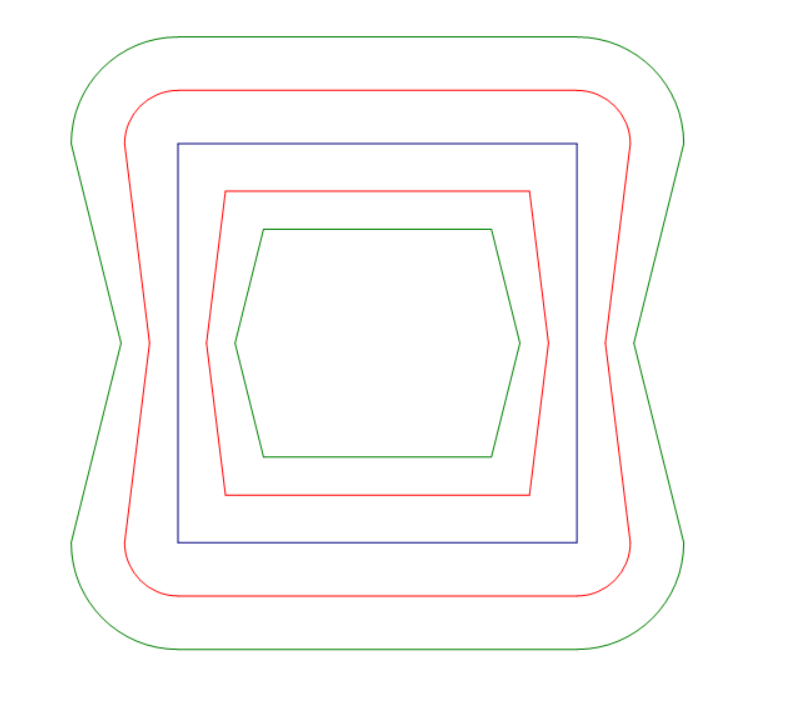}
  \caption{The blue line: $\partial \Omega$, the red line: $\partial W$, the green line: $\partial \widetilde{W}$}
  \end{figure}
    We note that $\Omega \cup \widetilde{W} \Subset [-3,3]+i[-2,2]$, now we explain why we choose $\delta=200\epsilon=4000c(g)h$.
    \begin{lemma}\label{eq15}
      For any $z \in \widetilde{W},$
      \begin{equation}\label{eq48}
      D(z, 2\widetilde{r}(z)) \cap \left \{w: \Re w \notin I_{\delta-\epsilon},\ |\Im w| \leq \frac{\epsilon}{4}\right \}= \emptyset,
      \end{equation}
      thus we have:
      \begin{equation}\label{eq21}
        ||(z-P^{\delta}_{\omega})^{-1}||\leq \frac{1}{\widetilde{r}(z)}, \ \text{for any} \ z \in \widetilde{W}.
      \end{equation}
      It follows that $P^{\delta}_{\omega}$ has no spectrum in $\widetilde{W}$. 
    \end{lemma}
    \begin{proof}
      For any $z\in \widetilde{W}$ and any $x \in D(z, 2\widetilde{r}(z))$, then there exists $y \in \gamma=\partial \Omega$ such that $|z-y|\leq \widetilde{r}(y)$. Since $4|\widetilde{r}(z)-\widetilde{r}(y)| \leq |z-y|\leq \widetilde{r}(y),$ it implies $\frac{3}{4}\widetilde{r}(y)\leq \widetilde{r}(z) \leq \frac{5}{4}\widetilde{r}(y).$ Then $$|x-y|\leq 2\widetilde{r}(z)+\widetilde{r}(y)\leq \frac{7}{2}\widetilde{r}(y).$$
      \begin{itemize}
      \item If $\Im y < 200\epsilon <\frac{b-a}{2}$ then $\Re y=a \ \text{ or } \ b$ by $y\in \gamma$, and $\widetilde{r}(y)\leq 4\epsilon+\frac{200\epsilon}{4}=54\epsilon$. It implies $|\Re x -\Re y|\leq 189\epsilon$, then $|\Re x -a |, |\Re x -b |\leq 189\epsilon < \delta-\epsilon$. It implies $\Re x \in I_{\delta-\epsilon}$.\\
      \item If $|\Im y|\geq 200\epsilon$, then $|\Im y|-|\Im x|\leq |\Im x-\Im y|\leq \frac{7}{2}\widetilde{r}(y)=14\epsilon+\frac{7}{8}|\Im y|.$ It implies $|\Im x|\geq \frac{1}{8}|\Im y|-14\epsilon \geq 11\epsilon$. 
      \end{itemize}
      Thus for any $z\in \widetilde{W}$ and any $x \in D(z, 2\widetilde{r}(z))$, we have $x \notin \left \{w: \Re w \notin I_{\delta-\epsilon},\ |\Im w| \leq \frac{\epsilon}{4}\right \}$, so we prove \eqref{eq48}. \\
      \indent Now we prove \eqref{eq21}. For any $z \in \widetilde{W}$, and any $z^{\prime} \in \sigma(P^{\delta})$, $\Re z^{\prime} \notin I_{\delta}$, then $z^{\prime}  \in \left \{w: \Re w \notin I_{\delta-\epsilon},\ |\Im w| \leq \frac{\epsilon}{4}\right \}$. By \eqref{eq48}, we have $d(z^{\prime},z)>2\widetilde{r}(z)$, it implies $d(z,\sigma(P^{\delta}))\geq 2\widetilde{r}(z)>\epsilon$, thus we have $$||(z-P^{\delta})^{-1}||\leq \frac{1}{2\widetilde{r}(z)}, \ \text{for any} \ z \in \widetilde{W}.$$ Since $(z-P^{\delta}_{\omega})(z-P^{\delta})^{-1}=\Id-K$, where $$K=ihQ_{\omega}(z-P^{\delta})^{-1}-h^{2}|\omega|_{x}^{2}(z-P^{\delta})^{-1},$$ then we use the previous argument in \eqref{eq3} to find $||K||<\frac{1}{2}$ for $c(g)h<10^{-4}$. Then by Neumann series, we prove \eqref{eq21}.
    \end{proof}
    In what follows, we follow the proof of \cite[Theorem 14.1.1]{SJ2} to construct the trace class perturbation operator by the semiclassical pseudodiffential calculus on the compact manifold, see Dyatlov-Zworski \cite[Appendix E]{DZ19}. Let $p(x,\xi)=|\xi|_{x}^{2}$ be the principal symbol of $P$, see \cite[Proposition E.14]{DZ19}. We introduce $\chi \in C_{0}^{\infty}(T^{*}X,[0,3])$ as
    $$\chi(x,\xi)= \begin{cases} 3 & \text { if } (x,\xi) \in p^{-1}([0,4]) \\ 0 & \text { if } (x,\xi) \in p^{-1}([5,\infty)) \end{cases}.$$
    Then we use a quantization procedure $\mathrm{Op}_{h}$ on the compact manifold $X$ to define $$\widetilde{P}:=P+i\mathrm{Op}_{h}(\chi), \quad \widetilde{P}_{\omega}:=P_{\omega}+i \mathrm{Op}_{h}(\chi).$$
    Here we briefly introduce the construction of the quantization procedure $\mathrm{Op}_{h}$, see \cite[Proposition E.15]{DZ19}. On the Euclidean space $\mathbb{R}^{d}$, for any $a \in C_{0}^{\infty}(T^*\mathbb{R}^d)$, the operator $\mathrm{Op}_{h}^{\mathrm{KN}}(a)\colon C^{\infty}(\mathbb{R}^{d})\mapsto C^{\infty}(\mathbb{R}^{d})$ which is defined by: $$\mathrm{Op}_{h}^{\mathrm{KN}}(a)u(x):=\left(\frac{1}{2\pi h}\right)^{d}\int_{\mathbb{R}^{2d}}e^{\frac{i(x-y)\cdot \xi}{h}}a(x,\xi)u(y)d\xi dy, \quad u\in  C^{\infty}(\mathbb{R}^{d}).$$
    Now we define a (non-canonical) quantization procedure on a closed surface $X$:
    \begin{definition}
    Let $\left(\varphi_j \colon U_j \mapsto V_{j}\subset \mathbb{R}^{2}, \chi_j\right)$ be cutoff charts of $X$, i.e. $\chi_{j}\in C_{0}^{\infty}(U_j)$, and $\chi_j^{\prime}$ functions satisfying $\chi_j^{\prime}=1$ near $\mathrm{supp}(\chi_j)$, then for any $a \in C_{0}^{\infty}(T^*X)$, the operator $\mathrm{Op}_{h}(a)\colon C^{\infty}(X)\mapsto C^{\infty}(X)$ which is defined by: 
$$
\mathrm{Op}_h(a):=\sum_{j} \chi_j^{\prime} \varphi_j^* \operatorname{Op}_h^{\mathrm{KN}}\left(\left(\chi_j a\right) \circ \widetilde{\varphi}_j^{-1}\right)\left(\varphi_j^{-1}\right)^* \chi_j^{\prime} .
$$
Here $\widetilde{\varphi}_j \colon T^*U_j \mapsto T^*V_j\subset T^*\mathbb{R}^2$ is defined by $\widetilde{\varphi}_{j}(x,\xi)=(\varphi_{j}(x), ((d\varphi_{j})_{x}^{t})^{-1}(\xi))$ which is the symplectic lifting of $\varphi_{j}$.
    \end{definition}
    Since $\chi \in C_{0}^{\infty}(T^{*}X,[0,3])$, then $\mathrm{Op}_{h}(\chi) \in \Psi_{h}^{k}(M)$ for all $k\in \mathbb{R}$, see \cite[Proposition E.12]{DZ19}. By Theorem \cite[Proposition E.22]{DZ19}, $\mathrm{Op}_{h}(\chi)\colon L^{2}(X) \mapsto H^{k}_{h}(X)$ is bounded uniformly in $h$ for all $k\in \mathbb{R}$. Then \cite[Proposition B.21]{DZ19} implies that $$\widetilde{P}-P \colon L^{2}(X) \mapsto L^{2}(X) \text { is a trace class operator }, \text { i.e. } ||\widetilde{P}-P||_{\mathrm{tr}}<\infty.$$ Since 
    $\sigma_{h}(\widetilde{P}_{\omega}-z)=p(x,\xi)+i \chi(x,\xi)-z$ satisfies that for any $z \in [-3,3]+i[-2,2]$:
    $$|\sigma_{h}(\widetilde{P}_{\omega}-z)|=\sqrt{(p(x,\xi)-\Re z)^2+(\chi(x,\xi)-\Im z)^2}\geq 1.$$
    By Theorem \cite[Proposition E.32]{DZ19} and \cite[Proposition E.22]{DZ19}, $\widetilde{P}_{\omega}-z \in \Psi^2(X)$ is a family holomorphic elliptic opertor for $z \in (-3,3)+i(-2,2)$, then the resolvent $(z-\widetilde{P}_{\omega})^{-1}\colon L^{2}(X) \mapsto L^{2}(X)$ is uniformly bounded in $z \in [-3,3]+i[-2,2]$. By \cite[Proposition B.28]{DZ19}, the eigenvalues of $P_{\omega}$ in $[-3,3]+i[-2,2]$ are the zeros of the determinant of a trace class perturbation of the identity, namely,
    \begin{equation*}
    D_{\omega}(z):=\det((P_{\omega}-z)(\widetilde{P}_{\omega}-z)^{-1})=\det(\Id-(\widetilde{P}-P)(\widetilde{P}_{\omega}-z)^{-1}).
    \end{equation*}
    Similarly, the eigenvalues of $P^{\delta}_{\omega}$ in $z \in [-3,3]+i[-2,2]$ are the zeros of 
    \begin{equation*}
      D_{\omega}^{\delta}(z):=\det((P^{\delta}_{\omega}-z)(\widetilde{P}_{\omega}-z)^{-1})=\det(\Id-(\widetilde{P}-P+P-P^{\delta})(\widetilde{P}_{\omega}-z)^{-1}),
    \end{equation*}
    Then we find that there are no such zeros of $D_{\omega}^{\delta}$ in $\widetilde{W}$ by \eqref{eq21} and the zeros of $D_{\omega}(z)$ and $D_{\omega}^{\delta}(z)$ are discrete for $z \in [-3,3]+i[-2,2]$. Now we briefly explain why $D_{\omega}(z)$ and $D_{\omega}^{\delta}(z)$ are holomorphic in $z \in (-3,3)+i(-2,2)$. \\
    \indent By \cite[Proposition B.29]{DZ19}, $D_{\omega}(z)$ and $D_{\omega}^{\delta}(z)$ are continuous in $[-3,3]+i[-2,2]$. Furthermore, by \cite[Page 511]{DZ19}, for a holomorphic family of operators $A(z)$ in trace class with $z$ in some domain $\Omega \subset \mathbb{C}$. The determinant $\operatorname{det}(I-A(z))$ is a holomorphic function of $z$, and for $z$ such that $I-A(z)$ is invertible. So we obtain that $D_{\omega}(z)$ and $D_{\omega}^{\delta}(z)$ are holomorphic in $(-3,3)+i(-2,2)$ except at the zeros. By Riemann's Removable Singularity Theorem, they are holomorphic in $z \in (-3,3)+i(-2,2)$.
    \begin{remark}
      We should point out that the determinants $D_{\omega}(z)$ and $D_{\omega}^{\delta}(z)$ may depend on the construction of $\chi$ and $\mathrm{Op}_{h}(\chi)$. But in the following context, we find $D_{\omega}(z)/D_{\omega}^{\delta}(z)$ is independent with $\chi$ and $\mathrm{Op}_{h}(\chi)$.
    \end{remark}
    We have 
    \begin{equation*}
      \begin{aligned}
        \frac{D_{\omega}(z)}{D^{\delta}_{\omega}(z)}&=\det((P_{\omega}-z)(\widetilde{P}_{\omega}-z)^{-1}(\widetilde{P}_{\omega}-z)(P^{\delta}_{\omega}-z)^{-1})=\det((P_{\omega}-z)(P^{\delta}_{\omega}-z)^{-1})\\
                                                    &=\det(1-(P^{\delta}-P)(P^{\delta}_{\omega}-z)^{-1}).\\
      \end{aligned}
    \end{equation*}
    Thus, by \cite[Proposition B.29]{DZ19},
    $$\left|\frac{D_{\omega}(z)}{D_{\omega}^{\delta}(z)}\right|\leq \exp ||(P^{\delta}-P)(P^{\delta}_{\omega}-z)^{-1}||_{\mathrm{tr}},$$
    and $$||(P^{\delta}-P)(P^{\delta}_{\omega}-z)^{-1}||_{\mathrm{tr}}\leq ||(P^{\delta}_{\omega}-z)^{-1}||\cdot||P^{\delta}-P||_{\mathrm{tr}}.$$
    From \eqref{eq1} and \eqref{eq21}, for any $z \in \widetilde{W}$:
    $$||(P^{\delta}-P)(P^{\delta}_{\omega}-z)^{-1}||_{\mathrm{tr}}\leq \frac{N^{P}_{X}(\lambda \in I_{\delta})\delta}{\widetilde{r}(z)}.$$
    We conclude that for any $z \in \widetilde{W}$:
    \begin{equation}\label{eq5}
      \log |D_{\omega}(z)|-\log |D^{\delta}_{\omega}(z)|\leq \frac{N^{P}_{X}(\lambda \in I_{\delta})\delta}{\widetilde{r}(z)}.
    \end{equation}
    Now we consider $\log |D^{\delta}_{\omega}(z)|-\log |D_{\omega}(z)|$, where 
    \begin{equation*}
      \begin{aligned}
        \frac{D^{\delta}_{\omega}(z)}{D_{\omega}(z)}=\det(1-(P^{\delta}-P)(P_{\omega}-z)^{-1}).
      \end{aligned}
    \end{equation*}
    Now we estimate $||(P_{\omega}-z)^{-1}||$ for $|\Im z|\geq \frac{\epsilon}{4}.$
    $$(z-P_{\omega})(z-P)^{-1}=\Id-ih Q_{\omega}(z-P)^{-1}+h^{2}|\omega|_{x}^{2}(z-P)^{-1},$$
    and $$||ih Q_{\omega}(z-P)^{-1}+h^{2}|\omega|_{x}^{2}(z-P)^{-1}||\leq h\cdot 2c(g) \cdot 4\epsilon^{-1} + c(g)^{2}h^{2}\cdot 4\epsilon^{-1}<\frac{1}{2},$$
    thus we have
    $$||(P_{\omega}-z)^{-1}||\leq \frac{2}{|\Im z|}\lesssim \frac{1}{\widetilde{r}(z)}, \ \text{ for } \ |\Im z| \geq \frac{\epsilon}{4}.$$
    Then we find for $|\Im z| \geq \frac{\epsilon}{4}$ and $z \in \widetilde{W}$:
    \begin{equation}\label{eq6}
      \log |D^{\delta}_{\omega}(z)|-\log |D_{\omega}(z)| \lesssim  \frac{N^{P}_{X}(\lambda \in I_{\delta})\delta}{\widetilde{r}(z)}.
    \end{equation}
    Recall that $D_{\omega}^{\delta}(z)$ is holomorphic in $(-3,3)+i(-2,2)$, let $$\log |D^{\delta}_{\omega}(z)|:=\frac{\Phi_{0}(z)}{h^{2}}.$$
    Then $\Phi_{0}(z)$ is harmonic function in $(-3,3)+i(-2,2)$ except at the zeros of $D_{\omega}^{\delta}(z)$. In particular, $\Phi_{0}(z)$ is harmonic in $\widetilde{W}$ as shown by \eqref{eq21}. It is known that if $u$ is holomorphic in a domain $\Omega \subset \mathbb{C}$, then $$\# \{z \in \Omega: u(z)=0 \}=\frac{1}{2\pi}\int_{\Omega}\log |u(z)|dA(z)$$ by $\Delta \log |z|=\delta_{0}$ in $\mathcal{D}^{\prime}(\mathbb{C})$, thus we have:
    \begin{equation}\label{eq16}
      N^{P^{\delta}_{\omega}}_{X}(\lambda \in \Omega)=\frac{1}{2\pi h^{2}}\int_{\Omega}\Delta \Phi_{0}(z)dA(z).
    \end{equation}
    For $z\in W$, we define $$\Phi(z)=\Phi_{0}(z)+\frac{C^{\prime }N^{P}_{X}(\lambda \in I_{\delta})\delta h^{2}}{(|\Im z|^{2}+\delta^{2})^{\frac{1}{2}}},$$
    where $C^{\prime }$ is a sufficiently large constant such that 
    \begin{equation}\label{eq47}
    \log |D_{\omega}(z)|\leq \frac{\Phi(z)}{h^{2}}, \quad \text{for} \ z \in W.
    \end{equation}
    Now we need the following lemma:
    \begin{lemma}
      Let $\gamma$ be a simple closed curve and $\widetilde{r} \colon \gamma \mapsto (0,\infty)$ be a Lipschitz function of Lipschitz modulus $\leq \frac{1}{4}$, fix $\theta>0$ and we define $$\widetilde{W}=\bigcup_{z \in \gamma}D(z,\widetilde{r}(z)), \quad W=\bigcup_{z \in \gamma}D(z,r(z)),$$
      where $r(z)=\frac{\widetilde{r}(z)}{1+\theta}$, then there exists $\Psi \in C^{\infty}_{0}(\mathbb{C},[0,1])$ such that 
      \begin{itemize}
      \item $\mathrm{supp}(\Psi)\subset \widetilde{W}$; 
      \item $\Psi=1$ on $W$; 
      \item $|\partial_{\Re z, \Im z}^{\alpha}\Psi(z)|\lesssim_{\alpha}r(z)^{-|\alpha|}.$ 
      \end{itemize}
    \end{lemma}
    \begin{proof}
      Let $$W_{\varepsilon}:=\bigcup_{z \in W}D(z,\varepsilon r(z)),$$
      where $\varepsilon>0$ satisfies $\frac{10\varepsilon}{4-\varepsilon}<\theta$, then for any $z_{1} \notin \widetilde{W}$, and any $|z_{2}-z_{1}|\leq \varepsilon r(z_{1})$ then $z_{2} \notin W_{\varepsilon}.$ If $z_{2} \in W_{\varepsilon}$, then there exists $z_{3} \in W$ and $z_{4} \in \gamma$ such that
      $$|z_{1}-z_{2}|<\varepsilon r(z_{1}), \quad |z_{2}-z_{3}|<\varepsilon r(z_{3}), \quad |z_{3}-z_{4}|<r(z_{4}). $$
      Then $$|z_{1}-z_{4}|<\varepsilon(r(z_{1})+r(z_{3}))+r(z_{4}),$$
      and $$r(z_{1})< \frac{4}{4-\varepsilon}r(z_{2}), \quad r(z_{2})< \frac{4+\varepsilon}{4}r(z_{3}), \quad r(z_{3})<\frac{5}{4}r(z_{4}).$$
      It implies
      $$|z_{1}-z_{4}|<\left ( \frac{10\varepsilon}{4-\varepsilon}+1 \right )r(z_{4})<(1+\theta)r(z_{4})=\widetilde{r}(z_{4}).$$
      Then $z_{1}\in \widetilde{W}$, it is a contradiction. So we prove that $z_{1}\notin \widetilde{W} $. Let $r_{s}(x)$ be
      $$
       r_{s}(z)=\sqrt{16\epsilon^{2}+\frac{|\Im z|^{2}}{16}}\asymp r(z).
      $$
      Now we construct $$\Psi(z)=\int_{\mathbb{C}} \widetilde{\chi}\left(\frac{z-w}{\varepsilon r_{s}(z)}\right)(\varepsilon r_{s}(z))^{-2}\mathbf{1}_{W_{\varepsilon}}(w)dA(w),$$
      where $\widetilde{\chi}\in C_{0}^{\infty}(D(0,1);\mathbb{R})$ satisfies $\widetilde{\chi}\geq 0$ and $\int_{\mathbb{C}}\widetilde{\chi}(z)dA(z)=1$. If $z \in W$, then we find $|w-z|\leq \varepsilon r_{s}(z)\leq \varepsilon r(z)$ implies $w \in W_{\varepsilon}$, then $\Psi(z)=1$. If $z \notin \widetilde{W}$, $|w-z|\leq \varepsilon r_{s}(z)\leq \varepsilon r(z)$ implies $w \notin W_{\varepsilon}$, then $\Psi(z)=0$. Finally,
      $$|\partial_{\Re z, \Im z}^{\alpha}\Psi(z)|\lesssim_{\alpha} r_{s}(z)^{-|\alpha|}\lesssim_{\alpha} r(z)^{-|\alpha|}.$$
    \end{proof}
    In our situation, $\theta=1$. Then we extend $\Phi$ to the whole $\overline{\widetilde{W}\cup \Omega}$ by 
    \begin{equation}\label{extendofharmonic}
    \Phi(z)=\Phi_{0}(z)+\frac{C^{\prime }N^{P}_{X}(\lambda \in I_{\delta})\delta h^{2}}{(|\Im z|^{2}+\delta^{2})^{\frac{1}{2}}}\Psi(z).
    \end{equation}
    It concludes that $$|\Delta \Phi(z)-\Delta \Phi_{0}(z)|\lesssim \frac{N^{P}_{X}(\lambda \in I_{\delta})\delta h^{2}}{(|\Im z|+\delta )^{3}}, \quad \forall z \in \widetilde{W}.$$
    Since $D_{\omega}^{\delta}$ has no zeros in $\widetilde{W}$, then $\Phi_{0}(z)$ is harmonic on $\widetilde{W}$, then
    $$\int_{\widetilde{W}}|\Delta \Phi(z)|dA(z)=\int_{\widetilde{W}}|\Delta \Phi(z)-\Delta \Phi_{0}(z)|dA(z).$$
    Now we estimate the integral 
    \begin{equation*}
      \begin{aligned}
        &\int_{\widetilde{W}}\frac{1}{(|\Im z|+\delta)^{3}}dA(z) \leq \int_{\widetilde{W} \cap \{|\Im z| \leq \frac{1}{2}\}}\frac{1}{(|y|+\delta)^{3}}dxdy+2\int_{[-3,3] \times [0,2]} \frac{1}{(\frac{1}{2}+\delta)^{3}}dxdy.\\
        &\leq 2\int_{0}^{\frac{1}{2}}\int_{a-4\epsilon-\frac{y}{4}}^{a+4\epsilon+\frac{y}{4}}\frac{1}{(y+\delta)^{3}}dxdy+ 2\int_{0}^{\frac{1}{2}}\int_{b-4\epsilon-\frac{y}{4}}^{b+4\epsilon+\frac{y}{4}}\frac{1}{(y+\delta)^{3}}dxdy+192\\
        &=4\int_{0}^{\frac{1}{2}}\frac{8\epsilon+y/2}{(y+\delta)^{3}}dy+192 \lesssim \delta^{-1}.\\
      \end{aligned}
    \end{equation*}
    The last inequality holds for $c(g)h<10^{-4}$. It implies
    \begin{equation}\label{eq7}
      \int_{\widetilde{W}}|\Delta \Phi(z)|dA(z)\lesssim N^{P}_{X}(\lambda \in I_{\delta})h^{2}.
    \end{equation}
    \subsection{Proof of Theorem \ref{perturoflaplacian}}\label{SS4}
    Now we apply Theorem \ref{holodistri}, taking $N$ distinct points $z^{0}_{j} \in \partial \Omega$, $j\in \mathbb{Z} / N \mathbb{Z}$ with the properties:
    \begin{itemize}
      \item They are distributed in the positive sense so that $j \mapsto \arg (z_{j}^{0}-\frac{a+b}{2})$ is increasing, and $z_{0}^{0}=b+i\frac{\epsilon}{2}$,
      \item There are precisely four points $z_{j}^{0}$ that minimize the distance to $\mathbb{R}$, namely $a\pm i \frac{\epsilon}{2},$ $b\pm i\frac{\epsilon}{2}$, 
      \item $\frac{r(z_{j}^{0})}{4} \leq |z^{0}_{j+1}-z^{0}_{j}|\leq \frac{r(z_{j}^{0})}{2}$.
    \end{itemize}
    Let $$N_{1}=\#\left\{z_{j}^{0}: \Re z_{j}^{0}=b, 0 \leq \Im z_{j}^{0} \leq \frac{b-a}{2}\right\}, \ N_{2}=\#\left\{z_{j}^{0}: a\leq \Re z_{j}^{0} \leq b, \Im z_{j}^{0}=\frac{b-a}{2}\right\}.$$ Since $$\Im z^{0}_{0}=\frac{\epsilon}{2}, \quad \Im z_{j+1}^{0}-\Im z_{j}^{0} \geq \frac{1}{8} \left (4\epsilon+\frac{\Im z_{j}^{0}}{4} \right ), \quad \forall j\leq N_{1},$$
    then $$\Im z_{j}^{0} \geq \left ( \frac{33}{32} \right )^{j}\cdot \frac{33\epsilon}{2}-16\epsilon \geq \frac{1}{2}\left ( \frac{33}{32} \right )^{j} \epsilon, \quad \forall j\leq N_{1}.$$
    Since $\Im z_{N_{1}}^{0} \leq \frac{b-a}{2}\leq \frac{3}{2}$, then $$N_{1}\leq 36|\log \epsilon|,$$
    and $$N_{2}\leq \frac{b-a}{\frac{\epsilon}{2}+\frac{b-a}{64}} < 64.$$
    By \eqref{eq6} and \eqref{extendofharmonic}, we can choose $$\epsilon_{j}=\frac{C^{\prime \prime}N^{P}_{X}(\lambda \in I_{\delta})\delta h^{2}}{|\Im z_{j}|+\delta},$$
    where $C^{\prime \prime}$ is a large constant such that for any $|\Im z|\geq \frac{\epsilon}{4}$ and $z\in W$,
    \begin{equation}\label{eq17}
    \log |D_{\omega}(z)| \geq \frac{1}{h^{2}}(\Phi(z)-\epsilon_{j}).
    \end{equation}
    By our construction, 
    \begin{equation}\label{eq24}
      \begin{aligned}
        \sum_{j \in \mathbb{Z} / N \mathbb{Z}}\epsilon_{j} 
        & \leq 4\sum_{j=0}^{N_{1}}\frac{C^{\prime \prime}N^{P}_{X}(\lambda \in I_{\delta})\delta h^{2}}{\Im z_{j}^{0}+\delta}+2\sum_{j=1}^{N_{2}} \frac{C^{\prime \prime}N^{P}_{X}(\lambda \in I_{\delta})\delta h^{2}}{\frac{b-a}{2}+\delta}\\
        & \leq C^{\prime \prime}N^{P}_{X}(\lambda \in I_{\delta})\delta h^{2} \left (\sum_{j=0}^{N_{1}} \left (\frac{33}{32} \right )^{-j} \epsilon^{-1}+N_{2}\delta^{-1} \right ) \lesssim N^{P}_{X}(\lambda \in I_{\delta})h^{2}.
      \end{aligned}
    \end{equation}
    For $z \in D\left (z_{j}^{0},\frac{r(z_{j}^{0})}{16} \right )\subset W$, we have $|\Im z|\geq  \frac{\epsilon}{2}-\frac{2\epsilon+\frac{\epsilon}{16}}{16} \geq \frac{\epsilon}{4}$. It implies \eqref{eq17} for any $z \in D\left(z_{j}^{0},\frac{r(z_{j}^{0})}{16} \right)$. Now we apply Theorem \ref{holodistri}. Recall that $C_1=1$ for $\Omega$, there exists a universal constant $C_{1}^{0}>0$ such that for any $C_{2}\geq C_{1}^{0}+16$, which satisfies:
    \begin{itemize}
      \item For any $z \in W$, $\log |D_{\omega}(z)|\leq \frac{\Phi(z)}{h^2}$ by \eqref{eq47}.\\
      \item For any $j$ and any $\widetilde{z} \in  D\left (z_{j}^{0},\frac{r(z_{j}^{0})}{2C_{2}} \right )$, $\log |D_{\omega}(\widetilde{z})| \geq \frac{1}{h^{2}}(\Phi(\widetilde{z})-\epsilon_{j})$ by $C_{2}\geq 16$.
    \end{itemize}
    Then there is a constant $C_{3}$ in Theorem \ref{holodistri} such that
    \begin{equation}
      \begin{aligned}
     \left | N^{P_{\omega}}_{X}(\lambda \in \Omega)-\frac{1}{2\pi h^{2}}\int_{\Omega}\Delta \Phi(z)dA(z) \right |\leq \frac{C_3}{h^{2}} \left ( \int_{W}|\Delta \Phi(z)|dA(z)+\sum_{j \in \mathbb{Z} / N \mathbb{Z}}\epsilon_{j} \right ).
      \end{aligned}
    \end{equation}
    By \eqref{eq7} and \eqref{eq24}, we have:
    \begin{equation}\label{eq22}
      \begin{aligned}
     \left | N^{P_{\omega}}_{X}(\lambda \in \Omega)-\frac{1}{2\pi h^{2}}\int_{\Omega}\Delta \Phi(z)dA(z) \right |\lesssim N^{P}_{X}(\lambda \in I_{\delta}).
      \end{aligned}
    \end{equation}
    Recall that: 
    \begin{equation*}
      \begin{aligned}
        &N^{P^{\delta}_{\omega}}_{X}(\lambda \in \Omega)=\frac{1}{2\pi h^{2}}\int_{\Omega} \Delta \Phi_{0}(z)dA(z)\\
        =&\frac{1}{2\pi h^{2}}\int_{\Omega} \Delta \Phi(z)dA(z) -\frac{1}{2\pi h^{2}}\int_{\Omega} \left ( \Delta \Phi- \Delta \Phi_{0} \right )dA(z),
      \end{aligned}
    \end{equation*}
    by \eqref{eq7}, we have:
    \begin{equation*}
      \begin{aligned}
      \left | \frac{1}{2\pi h^{2}}\int_{\Omega} \left ( \Delta \Phi- \Delta \Phi_{0} \right )dA(z) \right |  
      &\leq  \frac{1}{2\pi h^{2}}\int_{\widetilde{W}}|\Delta  \Phi- \Delta \Phi_{0}| dA(z) \lesssim N^{P}_{X}(\lambda \in I_{\delta}).
      \end{aligned} 
    \end{equation*}
    Then by \eqref{eq4},
    \begin{equation}\label{eq23}
    \left | N^{P}_{X}(\lambda \in I)-\frac{1}{2\pi h^{2}}\int_{\Omega}\Delta \Phi(z)dA(z) \right |\lesssim N^{P}_{X}(\lambda \in I_{\delta}).
    \end{equation}
    We prove Theorem \ref{perturoflaplacian} by \eqref{eq22} and \eqref{eq23}.
    \subsection{Proof of Theorem \ref{semiweyltwi}}\label{SS5}
     Finally, we estimate $N^{P}_{X}(\lambda \in I_{\delta})$. By Theorem \ref{uniweyl}, for any $g>g_{0},$
      \begin{equation*}
      \frac{N_{X}^{P}(b-\delta, b+\delta)}{\mathrm{Vol}(X)}=\frac{N_{X}^{-\Delta}((b-\delta)h^{-2}, (b+\delta)h^{-2})}{\mathrm{Vol}(X)} \leq \frac{2\delta h^{-2}}{4\pi}+R(X,(b-\delta)h^{-2},(b+\delta)h^{-2}),
      \end{equation*}
      and for $0<h\leq 1$ and $c(g)h<10^{-4}$, let $u=b+\delta+h^{2}<2$, then 
      \begin{equation*}
        \begin{aligned}
          R(X,(b-\delta)h^{-2},(b+\delta)h^{-2})
          &\leq C_{0}h^{-1}\sqrt{\frac{b+\delta+h^2}{\log g}\log \left( 2+2\delta h^{-1}\sqrt{\frac{\log g}{b+\delta+h^2}}\right )} \\
          & \leq C_{0}h^{-1}\sqrt{\frac{u\log (2u^{\frac{1}{2}}+8000   c(g)\sqrt{\log g})-\frac{1}{2}\cdot u\log u}{\log g}} \\
          & \lesssim h^{-1}\sqrt{\frac{\log (2+2c(g)\sqrt{\log g})}{\log g}}.
        \end{aligned}
      \end{equation*}
      Similarly,
      \begin{equation*}
      R(X,(a-\delta)h^{-2},(a+\delta)h^{-2})\lesssim h^{-1}\sqrt{\frac{\log (2+2c(g)\sqrt{\log g})}{\log g}}.
      \end{equation*}
      Here we simplify the upper bound, since
    $$\sqrt{\log (2+2c(g)\sqrt{\log g})}\leq \sqrt{1+2c(g)\sqrt{\log g}} \leq c(g)\sqrt{\log g}+1,$$
    then       
    \begin{equation*}
      \sqrt{\frac{\log 2}{\log g}}+c(g) \leq \sqrt{\frac{\log (2+2c(g)\sqrt{\log g})}{\log g}}+c(g)\leq 2c(g)+\sqrt{\frac{1}{\log g}}.
      \end{equation*}
      It implies
        \begin{equation}\label{constofgenus}
          \frac{1}{2}\left(\sqrt{\frac{1}{\log g}}+c(g)\right)\leq \sqrt{\frac{\log (2+2c(g)\sqrt{\log g})}{\log g}}+c(g)\leq 2\left(\sqrt{\frac{1}{\log g}}+c(g)\right).
          \end{equation}
      Then for $0<h\leq 1$, $c(g)h<10^{-4}$ and $g>g_{0}$:
      \begin{equation}\label{eq10}
        N^{P}_{X}(\lambda \in I_{\delta}) \lesssim \mathrm{Vol}(X)\left ( \sqrt{\frac{1}{\log g}}+c(g) \right ) h^{-1},
      \end{equation}
      By Theorem \ref{perturoflaplacian} and \eqref{eq10}, we finish the proof of Theorem \ref{semiweyltwi}.
    \section{$L^{\infty}$ norm of harmonic forms on typical surfaces}\label{sec4}
    Theorem \ref{semiweyltwi} demonstrates that the spectral distribution of $\Delta_{\omega}$ depends on the $L^{\infty}$ norm of $\omega$. Gilmore, Le Masson, Sahlsten, and Thomas \cite[Theorem 1.1]{GCLS} show that the $L^{\infty}$ norms of $L^2$-normalized eigenfunctions on $X$ decay to $0$ with high probability in the large genus limit. Inspired by their results, we show in the following theorem that the $L^{\infty}$ norms of harmonic forms have a similar estimate.
\begin{theorem}\label{deh}
  For any $0<\epsilon<\frac{1}{8}$, there exists $g_{\epsilon}>0$ depending only on $\epsilon$ such that for any $g>g_{\epsilon}$, $X \in \mathcal{L}_{g}$ defined in \eqref{short} and $\omega \in \mathcal{H}^{1}(X)$, we have
  \begin{equation}
  ||\omega||_{L^{\infty}}\leq \frac{g^{-\frac{1}{8}+\epsilon}}{\sqrt{\mathrm{Inj}(X)}}||\omega||_{L^{2}}.
  \end{equation}
\end{theorem}
    Brock and Dunfield \cite[Theorem 4.1]{BD} estimated $L^{\infty}$ norm of harmonic 1-form $\omega$ by its $L^{2}$ norm on a closed 3-hyperbolic manifold, and Han \cite{HH} generalized it to a finite 3-hyperbolic manifold:
    \begin{theorem}[Brock-Dunfield \cite{BD}]\label{injesti}
      If $\alpha$ is a harmonic 1-form on a closed 3-hyperbolic manifold $M$, then
    \begin{equation}\label{estiforinf}
    \|\alpha\|_{L^{\infty}} \lesssim \frac{1}{\sqrt{\operatorname{Inj}(M)}}\|\alpha\|_{L^{2}}
    \end{equation}
    \end{theorem}
    It is not difficult to find Theorem \ref{injesti} also holds for closed hyperbolic surfaces. However, Mirzakhani shows that for any $r>0$, $\mathrm{Inj}(X)\geq r>0$ does not occur with high probability in \cite{MM}. We try to improve the estimate in \eqref{estiforinf} on a typical surface. 
    \subsection{The operator on 1-form space generated by radial function}
      By the proof of \cite[Theorem 1.1]{GCLS}, we try to construct the operator $A$ on 1-form space generated by the radial function $k$, and we expect that the harmonic form, generally, an eigenform of Hodge Laplacian\footnote{Hodge Laplacian $\Delta=d\delta+\delta d$, $\delta$ is the dual of $d$, which is positive Laplacian. In this section, we assume $\Delta$ is Hodge Laplacian.} is also an eigenform of $A$. We turn to study the function related to the harmonic form $\omega$. For any $\omega \in \mathcal{H}^{1}(X)$, we can lift it into the contractible universal cover space $\mathbb{H}$\footnote{Here we use Poincaré half-plane model.} and we still denote it $\omega$. Then there is a smooth function $f$ on $\mathbb{H}$ such that $$\omega_{x}=df_{x}, \quad \forall x\in \mathbb{H}.$$
      Let $K \colon \mathbb{R}\mapsto \mathbb{R}$ be a compactly supported function then define $k \colon \mathbb{H}\times \mathbb{H} \mapsto \mathbb{R}$ by $k(x,y)=K(d(x,y))$.
      We define $$F(x)=\int_{\mathbb{H}}k(x,y)f(y)d\Vol(y),$$
      where $F$ is $\Gamma$-invariant. The operator $A_0$ on $\mathcal{H}^{1}(X)$ is defined by
      \begin{equation}\label{ophar}
      (A_{0}\omega)_{x}=dF_{x}, \quad \forall x\in \mathbb{H}.
      \end{equation}
      We notice that $A_0$ is well defined since when $df_{1}=df_{2}$ then $f_{1}-f_{2}=c$, $c$ is some constant and 
      $$\int_{\mathbb{H}}k(x,y)\cdot (f_{1}-f_{2})d\Vol(y)=2\pi c\int_{0}^{\infty}K(\rho)\sinh(\rho)d\rho$$
      which is a constant, so we deduce that $A_{0}\omega=dF$ is well defined.
      Now we denote the stable subgroup $$G_{x}:=\{\gamma \in \mathrm{Iso}(\mathbb{H}): \gamma x=x\},$$
      and $d\gamma$ is the normalized Haar measure on $G_{x}$ ($G_{x}$ is isomorphic to $\mathbb{S}^{1}$).
      So we consider
      \begin{equation}
        \begin{aligned}
          F(x)&=\int_{\mathbb{H}}k(x,y)f(y)d\Vol(y)=\int_{\mathbb{H}}\int_{G_{x}}k(\gamma x,y)f(y)d\gamma d\Vol(y)\\
              &=\int_{\mathbb{H}}\int_{G_{x}}k(x,\gamma^{-1}y)f(y)d\gamma d\Vol(y)=\int_{\mathbb{H}}k(x,y)\int_{G_{x}}f(\gamma y)d\gamma d\Vol(y)\\
        \end{aligned}
      \end{equation}
    We define $$\widetilde{f}_{x}(y)=\int_{G_{x}}f(\gamma y)d\gamma.$$
    When $\omega \in \mathcal{H}^{1}(X)$, $$\Delta f=\delta df=\delta \omega=0.$$
    Then $\widetilde{f}_{x}(y)$ is also harmonic, only depends on $d(x,y)$ and $\widetilde{f}_{x}(x)=f(x)$. 
    By \cite[Page 49 (3.5)]{BD16}, we have:
    $$\widetilde{f}_{x}(y)=f(x), \ \forall y \in \mathbb{H}.$$
    Then $$F(x)=\int_{\mathbb{H}}k(x,y)d\Vol(y)\cdot f(x)=2\pi\int_{0}^{\infty}K(\rho)\sinh(\rho)d\rho\cdot f(x).$$
    We deduce that 
    \begin{equation}\label{4radial}
      (A_{0}\omega)_{x}=dF_{x}=2\pi\int_{0}^{\infty}K(\rho)\sinh(\rho)d\rho\cdot df_{x}=2\pi\int_{0}^{\infty}K(\rho)\sinh(\rho)d\rho\cdot \omega_{x}.
    \end{equation}
    \subsection{$L^{\infty}$ norm of harmonic 1-form}
    Now we take $k(x,y)=K(d(x,y))$ as the test function in \cite[Section 4.2]{GCLS}, 
    \begin{equation}\label{untermpertest}
    K_{t}(\rho)=\frac{1}{\sqrt{\cosh(t)}}\mathbf{1}_{\{\rho<t\}}, \quad k_{t}(x,y)=K_{t}(d(x,y)).
    \end{equation}
    Then
    \begin{equation}\label{4inf}
     2\pi\int_{0}^{\infty}K_{t}(\rho)\sinh(\rho)d\rho=\frac{2\pi}{\sqrt{\cosh(t)}}\int_{0}^{t}\sinh(\rho)d\rho=2\pi\frac{\cosh(t)-1}{\sqrt{\cosh(t)}}.
    \end{equation}
    It deduces that
    \begin{equation}\label{4eq}
      (A_{0}\omega)_{x}=2\pi\frac{\cosh(t)-1}{\sqrt{\cosh(t)}}\omega_{x}.
    \end{equation}
    Now we study the $dF_{x}(v)$ for any $v \in T_{x}\mathbb{H}$, $|v|_{x}=1$. Since the isometric group $\mathrm{Iso}(\mathbb{H})$\footnote{In fact, $\mathrm{Iso}(\mathbb{H})$ is isomorphic to $\mathrm{PSL}(2,\mathbb{R})$, the action is given by Möbius transformations, see e.g. \cite{BU} and \cite{BD16}.} of $\mathbb{H}$ acts on the sphere bundle $S\mathbb{H}$ is transitive, i.e. for any $(x,v) \in S\mathbb{H}$, there exists $\gamma \in \mathrm{Iso}(\mathbb{H})$ such that $\gamma \cdot (i,(0,1))=(\gamma(i),(d\gamma)_{i}(0,1))$. So we can assume that $$x=i, \quad v=(0,1).$$
    We construct a family isometries $\varphi_{t} \in \mathrm{Iso}(\mathbb{H}), t\in \mathbb{R}$ such that
    $$\varphi_{t}(z)=e^{t}z, \quad \forall z \in \mathbb{H}.$$ 
    We obtain $$\left. \frac{d\varphi_{t}}{dt} \right|_{t=0}(i)=(0,1).$$
    It deduces that 
    \begin{equation}
      dF_{x}(v)=\lim_{t\to 0}\frac{F(\varphi_{t}(x))-F(x)}{t}.
    \end{equation}
    Now we consider $F(\varphi_{t}(x)),$
    \begin{equation}
      \begin{aligned}
      F(\varphi_{t}(x))&=\int_{\mathbb{H}}k(\varphi_{t}(x),y)f(y)d\Vol(y)=\int_{\mathbb{H}}k(x,\varphi_{t}^{-1}(y))f(y)d\Vol(y)\\
      &=\int_{\mathbb{H}}k(x,y)f(\varphi_{t}(y))d\Vol(y).\\
      \end{aligned}
    \end{equation}
    Then we deduce that\footnote{Here $y$ is identified with the tangent vector at $T_{y}\mathbb{H}$.}
    \begin{equation}
      \begin{aligned}
      \lim_{t\to 0}\frac{F(\varphi_{t}(x))-F(x)}{t}&=\lim_{t \to 0}\frac{1}{t} \left ( \int_{\mathbb{H}}k(x,y)f(\varphi_{t}(y))d\Vol(y)-\int_{\mathbb{H}}k(x,y)f(y)d\Vol(y) \right )\\
                                                   &=\int_{\mathbb{H}}k(x,y)\lim_{t \to 0}\frac{f(e^{t}y)-f(y)}{t}d\Vol(y)\\
                                                   &=\int_{\mathbb{H}}k(x,y)df_{y}(y)d\Vol(y)=\int_{\mathbb{H}}k(x,y)\omega_{y}(y)d\Vol(y).
      \end{aligned}
    \end{equation}
    Here $$|y|_{y}^{2}=\frac{1}{y^{2}}y^{2}=1.$$
    Then
    \begin{equation}\label{calculate1}
      \begin{aligned}
        dF_{x}(v)=\int_{\mathbb{H}}k(x,y)\omega_{y}(y)d\Vol(y)=\int_{D}\sum_{\gamma \in \Gamma}k(x,\gamma y)\omega_{\gamma y}(\gamma y)d\Vol(y).
      \end{aligned}
    \end{equation}
    Now we estimate $|dF|_{x}$, by Cauchy-Schwarz inequality:
    \begin{equation}\label{diffinteresti}
      \begin{aligned}
        |dF_{x}(v)|&\leq \int_{D} \sum_{\gamma \in \Gamma}k(x,\gamma y) \cdot |\omega|_{\gamma y}\cdot |\gamma y|_{\gamma y} d\Vol(y)\leq \int_{D} \sum_{\gamma \in \Gamma}k(x,\gamma y) \cdot |\omega|_{y}d\Vol(y)\\
                   &\leq \left ( \int_{D} \Big| \sum_{\gamma \in \Gamma}k(x,\gamma y) \Big|^{2}d\Vol(y) \right )^{\frac{1}{2}} \cdot ||\omega||_{L^{2}}.
      \end{aligned}
    \end{equation}
    In Gilmore-Le Masson-Sahlsten-Thomas \cite{GCLS}, they considered the event $\mathcal{A}_{g}^{b,c}$ defined by
    \begin{equation}\label{shortloop}
    \mathcal{A}_{g}^{b,c}=\{X \in \mathcal{M}_{g}: \mathrm{Inj}(X)\geq g^{-b}, N_{c\log g}(X)\leq 1\},
    \end{equation}
    where we denote by $N_{L}(X,x)$ the number of the primitive geodesic loops $\gamma$ of $\ell_{X}(\gamma) \leq L$ based at $x \in X$ then 
    \begin{equation*}
    N_{L}(X)=\sup_{x \in X}N_{L}(X,x).
    \end{equation*}
    The event $\mathcal{A}_{g}^{b,c}$ occurs with high probability when $b,c$ are small enough. Yunhui Wu and Yuhao Xue suggest to us that $\mathcal{L}_{g}$ defined in \eqref{short} can replace \eqref{shortloop} to obtain the explicit value of $c$ in \eqref{shortloop}. We have the following theorem:
    \begin{theorem}\label{count11}
      For any $X \in \mathcal{L}_{g}$ defined in \eqref{short}, for any $\epsilon>0$, then for large genus $g$:
    \begin{equation}\label{count1}
      N_{(\frac{1}{2}-\epsilon)\log g}(X)\leq 1.
    \end{equation}
    \end{theorem}
    \begin{proof}
      The event $\mathcal{L}_{g}$ is closely related to the tangle-free property studied by Monk-Thomas \cite{MT20}:\\
      \indent Let $X\in \mathcal{M}_{g}$ and $L>0$. We say that $X$ is $L$-tangle-free if all embedded pairs of pants and one-holed tori in $X$ have a total boundary length larger than $2 L$. Since the boundaries of embedded pants or one-holed tori form simple closed multi-geodesics separating $X$, then we deduce that for any $\epsilon>0$,
      \begin{equation}\label{tgl1}
        \mathcal{L}_{g}\subset TF_{g}^{\epsilon}:=\{X \in \mathcal{M}_{g}: X \ \text{is} \ (1-\epsilon)\log g \text{-tangle-free} \}
      \end{equation}
      holds for sufficiently large $g$. Consequently, $\lim\limits_{g \to \infty}\mathbb{P}^{\mathrm{WP}}_{g}(TF_{g}^{\epsilon})=1$. More precisely, Monk-Thomas \cite[Theorem 5]{MT20} proved that $1-\mathbb{P}^{\mathrm{WP}}_{g}(TF_{g}^{\epsilon})=\mathcal{O}((\log g)^{2} g^{-\epsilon})$.\\
    On the suggestion of Thomas, we use the following theorem \cite[Theorem 9]{MT20} to prove \eqref{count1}:
    \begin{theorem}[Monk-Thomas, \cite{MT20}]\label{MTthm9}
      If $x$ is a point on a L-tangle-free surface $X$, and $\delta_x$ is the shortest geodesic loop based at $x$, then any other loop $\beta$ based at $x$ such that $\ell\left(\delta_x\right)+\ell(\beta)<L$ is homotopic to a power of $\delta_x$.
    \end{theorem}
    For any $X \in \mathcal{L}_{g}$ defined in \eqref{short}, and any $\epsilon>0$, then for large genus $g$, $X$ is a $(1-\epsilon)\log g$-tangle-free surface by \eqref{tgl1}.
    If $N_{(\frac{1}{2}-\epsilon)\log g}(X)>1$, then there exist a point $x\in X$, and a primitive geodesic loop $\beta$ based at $x$ which is not homotopic to $\delta_x$ and $\delta_{x}^{-1}$ and satisfies $$\ell(\delta_{x})+\ell(\beta)\leq (1-2\epsilon)\log g<(1-\epsilon)\log g.$$
    It is a contradiction with Theorem \ref{MTthm9}, so we prove \eqref{count1}.
  \end{proof}

  \begin{remark}
    We can also apply the constructions in Nie-Wu-Xue \cite[Section 3]{NWX} and \cite[Proposition 59]{NWX} with a slight modification, and use the classical isoperimetric inequality (see \cite{BU} and \cite[Lemma 8, Lemma 9]{WX1}) to prove \eqref{count1}.
  \end{remark}

  \noindent Now we apply the following counting in \cite[Lemma 6.1]{GCLS}:
    \begin{lemma}[Gilmore-Le Masson-Sahlsten-Thomas, \cite{GCLS}]
    Suppose that $X=\Gamma \setminus \mathbb{H}$ is a closed hyperbolic surface for which there exists an $R>0$ such that $N_R(X) \leq n$. Then for each $z, w \in D$,
\begin{equation}\label{count2}
\left|\left\{\gamma \in \Gamma: d(z, \gamma w) \leq \frac{r}{2}\right\}\right| \leq \frac{2 n r}{\operatorname{Inj}(X)}+2, \quad \forall r \leq R.
\end{equation}
\end{lemma}
    We choose $R=(\frac{1}{2}-\epsilon)\log g$ and $t=\frac{R}{2},$ then we apply \eqref{count1} and \eqref{count2} to estimate the integral kernel:
    $$\#\{\gamma \in \Gamma: d(z,\gamma w)\leq t \}\leq \frac{4t}{\mathrm{Inj}(X)}+2.$$
    By the choice of $k_t$ in \eqref{untermpertest}, it implies that 
    \begin{equation}
      \begin{aligned}
        &\int_{D}\Big|\sum_{\gamma \in \Gamma}k(x,\gamma y)\Big|^{2}d\Vol(y)= \frac{1}{\cosh(t)}\int_{D} \sum_{\gamma \in \Gamma}\mathbf{1}_{d(x,\gamma y)\leq t}(y)d\Vol(y) \\
        \leq & \frac{1}{\cosh(t)} \cdot \left ( \frac{4t}{\mathrm{Inj}(X)}+2 \right )\mathrm{Vol}(B(x,t))\leq 2\pi \left ( \frac{4t}{\mathrm{Inj}(X)}+2 \right ).
      \end{aligned}
    \end{equation}
    By \eqref{diffinteresti}, we have: 
    \begin{equation*}
      |dF|_{x}=\sup_{|v|_{x}=1}|dF_{x}(v)|\leq \sqrt{2\pi \left ( \frac{4t}{\mathrm{Inj}(X)}+2 \right )} ||\omega||_{L^{2}}.
    \end{equation*}
    Since $\mathrm{Inj}(X) \leq 2 \log g$, we can remove the term $+2$ to obtain that 
    \begin{equation}\label{4geo}
      |dF|_{x} \leq \sqrt{ \frac{6\pi \log g}{\mathrm{Inj}(X)} } ||\omega||_{L^{2}}.
    \end{equation}
    Another side, we have proved for some $g>g_{\epsilon}^{\prime}$,
    \begin{equation}\label{4spe}
      |dF|_{x}=|A_{0}\omega|_{x}=2\pi\frac{\cosh(t)-1}{\sqrt{\cosh(t)}}\cdot|\omega|_{x}\gtrsim e^{\frac{t}{2}}|\omega|_{x}.
    \end{equation}
    Combinate \eqref{4geo} and \eqref{4spe}, we have
    \begin{equation*}
      ||\omega||_{L^{\infty}}\lesssim \frac{g^{-\frac{1}{8}+\frac{\epsilon}{4}}\sqrt{\log g}}{\sqrt{\mathrm{Inj}(X)}}||\omega||_{L^{2}}.
    \end{equation*}
    For sufficiently large $g>g_{\epsilon}$, we obtain Theorem \ref{deh}:
    \begin{equation}
      ||\omega||_{L^{\infty}}\leq \frac{g^{-\frac{1}{8}+\epsilon}}{\sqrt{\mathrm{Inj}(X)}}||\omega||_{L^{2}}.
    \end{equation}
    \subsection{Generalization to the eigenforms}
    In this subsection, we generalize Theorem \ref{deh} to general eigenforms of the Hodge Laplacian. An eigenform $\omega$ of $\lambda=\frac{1}{4}+r^2$ is defined by $\Delta \omega=\lambda \omega$. We have the following theorem:
    \begin{theorem}\label{delocaleigen}
      There exists $g_{0}>0$ then for any $g>g_{0}$ and $\lambda\geq \frac{1}{4}$, there exists a constant $C_\lambda$ which depends on $\lambda$ continuously such that for any $X \in \mathcal{L}_{g}$ in \eqref{short} and any eigenform $\omega_{\lambda}$ of the Hodge Laplacian associated with eigenvalue $\lambda$, we have:
    \begin{equation*}\label{termperform}
   ||\omega_{\lambda}||_{L^{\infty}} \leq C_{\lambda}\sqrt{\frac{1}{\mathrm{Inj}(X) \log g}} ||\omega_{\lambda}||_{L^2},
    \end{equation*}
    Moreover, for any $\epsilon>0$ and $0<\beta\leq \frac{1}{2}$, there exists $g_{\epsilon,\beta}>0$ depending on $\epsilon,\beta$ and $C_{\beta}$ depending continuously on $\beta$, then for any $g>g_{\epsilon,\beta}$ and $\lambda \in \left [0, \frac{1}{4}-\beta^{2}\right ]$ such that for any $X \in \mathcal{L}_{g}$ and any eigenform $\omega_{\lambda}$ of the Hodge Laplacian associated with eigenvalue $\lambda$, we have:
    \begin{equation*}\label{untermperform}
   ||\omega_{\lambda}||_{L^{\infty}} \leq C_{\beta}\frac{g^{-\frac{\beta}{4}+\epsilon}}{\sqrt{\mathrm{Inj}(X)}} ||\omega_{\lambda}||_{L^2}.
    \end{equation*}
    \end{theorem}
    We introduce how to construct the operator generated by the radial kernel. Recall the Hodge decomposition theorem, see \cite{WF}, for any $\omega \in \Omega^{1}(X)$, we have $\omega=\omega_{0}+df+\delta \eta,$ where $\omega \in \mathcal{H}^{1}(X)$, $f\in C^{\infty}(X)$ and $\eta \in \Omega^{2}(X)$. We define $g=\star \eta$, where $\star$ is the Hodge star operator. Since $\delta=-\star d \ \star$, $\delta \eta= -\star d \star \eta=-\star dg$. Let radial function $k:\mathbb{H}\times \mathbb{H} \mapsto \mathbb{R}$ be the same as the previous section, then we define
    \begin{equation}\label{radialfunction}
        F_{1}(x)=\int_{\mathbb{H}}k(x,y)f(y)d\Vol(y), \quad F_{2}(x)=\int_{\mathbb{H}}k(x,y)g(y)d\Vol(y),
    \end{equation}
    and $A_{0}$ is the operator which we defined in \ref{ophar}. Then we define $A:\Omega^{1}(X) \to \Omega^{1}(X)$:
    \begin{equation}\label{operatorform}
      (A\omega)_{x}=(A_{0}\omega_{0})_{x}+(dF_{1})_{x}-\star (dF_{2})_{x}.
    \end{equation}
    In the same way, we find $A\omega$ is independent with the choice of $f$ and $\eta$. Now we consider if $\omega$ is the eigenform of $\lambda=\frac{1}{4}+r^{2}>0,$ i.e. $\Delta \omega=\lambda \omega.$
    It implies $$\Delta \omega=d \Delta f+\delta \Delta \eta=\lambda \omega_{0}+\lambda df+\lambda \delta \eta.$$ By the orthogonality and $\lambda>0$, we obtain that $\omega_{0}=0$, $d \Delta f=\lambda df$, $\delta \Delta \eta=\lambda \delta \eta$.
    Since $\lambda>0$, we can choose $f$ and $\eta$ such that:
    \begin{equation}\label{eigenformtoeigenfunction}
        \Delta f=\lambda f, \quad \Delta g=\lambda g, \ \text{where} \ g=\star \eta.
      \end{equation}
    Then we find that $\omega=df+\delta \eta=df+\delta \star g=df-\star dg.$ By Iwaniec \cite[Theorem 1.14]{IWH21} and Le Masson-Sahlsten \cite[Proposition 2.1]{LS}, $$F_{1}(x)=\mathcal{S}(k)(r)f(x), \quad F_{2}(x)=\mathcal{S}(k)(r)g(x),$$ where $\mathcal{S}(k)$ is the Selberg transform of $k$ in Theorem \ref{Sel}, then we concludes that:
    \begin{equation}\label{radialeigenform}
      (A\omega)_{x}=\mathcal{S}(k)(r)(df_{x}-\star dg_{x})=\mathcal{S}(k)(r)\omega_{x}.
    \end{equation}
    We can prove Theorem \ref{delocaleigen} using the methods in \cite[Theorem 1.1]{GCLS}. The details for its proof will be provided in Appendix \ref{app2}.
    \section{Proof of the main theorems}\label{maintheoremsection}
    \subsection{Proof of Theorem \ref{uniweylfortwist}}
      By Theorem \ref{deh} and \eqref{geoassump}, we take $\epsilon=\frac{1}{50}$, then for sufficiently large $g$ such that for any $X \in \mathcal{A}_{g} \cap \mathcal{L}_{g}$, and $\omega \in \mathcal{H}^{1}(X,\mathbb{R})$, we have:
      $$||\omega||_{L^{\infty}}\leq \frac{g^{-\frac{1}{8}+\frac{1}{50}}}{\sqrt{\mathrm{Inj}(X)}}||\omega||_{L^2}, \quad \mathrm{Inj}(X)\geq g^{-\frac{1}{24}}(\log g)^{\frac{9}{8}}.$$
      It implies there exists a universal constant $g_{0}^{\prime}>0$ such that for any $g>g_{0}^{\prime}$, and $X \in \mathcal{A}_{g} \cap \mathcal{L}_{g}$:
      \begin{equation}\label{eq31}
      ||\omega||_{L^{\infty}}\leq g^{-\frac{1}{48}}||\omega||_{L^2}, \ \text{for any} \ \omega \in \mathcal{H}^{1}(X,\mathbb{R}).
      \end{equation}
      We take $C=10(C_{0}+C_{0}^{\prime})$ where $C_{0}$, $C_{0}^{\prime}$ are universal constant in Theorem \ref{uniweyl} and \ref{uniweylfortwist} respectively. For any $c>0$, then we choose $c(g)=cg^{-\frac{1}{48}}$. We choose a constant $\varepsilon=\min \left\{C_{0}, \ C_{0}^{\prime}, \ 1\right\}$. Then there exists a constant $g_{c}^{\prime}$ only depending on $c$ such that for any $g>g_{c}^{\prime}$, $2\varepsilon \sqrt{\frac{1}{\log g}}\geq 8000c(g)$. Let 
      \begin{equation}\label{eq37}
      g_{c}:=\left[g_{0}+g_{0}^{\prime}+g_{c}^{\prime}+\left((10^{4}c)^{48}\right)\right]+1,
      \end{equation}
      then for any $g>g_{c}$, any interval $I=[a,b]$ with $b\geq 0$, any $X \in \mathcal{A}_{g} \cap \mathcal{L}_{g}$, and $\omega \in \mathcal{H}^{1}(X,\mathbb{R})$ with $||\omega||_{L^2}\leq c$. By \eqref{eq31}, we have: $$||\omega||_{L^{\infty}}\leq cg^{-\frac{1}{48}}=c(g), \quad c(g)<10^{-4}.$$ By \eqref{principalvalue} and \eqref{estiofprin}, 
      \begin{equation}\label{eq38}
      \lambda_{0}^{\omega}(X) \geq -c(g)(1+c(g)) >-\frac{1}{2}.
      \end{equation}
      If $a<-2$, then $N^{-\Delta_{\omega}}_{X}(\Re \lambda \in [a,-2])=0$, so we assume that $a\geq -2$. \\
      \indent Let $h=(b+1)^{-\frac{1}{2}}\in (0,1]$, and define $a^{\prime}:=ah^{2}=\frac{a}{b+1}$, $b^{\prime}:=bh^{2}=\frac{b}{b+1}$, then $-2\leq a^{\prime} \leq b^{\prime} \leq 1$, $b^{\prime} \geq 0$ and $$N^{-\Delta_{\omega}}_{X}(\Re \lambda \in [a,b])=N^{P_{\omega}}_{X}(\Re \lambda \in [a^{\prime},b^{\prime}]), \quad N^{-\Delta}_{X}(\lambda \in [a,b])=N^{P}_{X}(\lambda \in [a^{\prime},b^{\prime}]).$$
      If $b^{\prime}-a^{\prime}\geq 2\varepsilon \sqrt{\frac{b+1}{\log g}}h \geq 8000c(g)h$, by Theorem \ref{semiweyltwi}:
      $$\left| N^{-\Delta_{\omega}}_{X}(\Re \lambda \in [a,b])-N^{-\Delta}_{X}(\lambda \in [a,b]) \right|\leq C_{0}^{\prime}\mathrm{Vol}(X)\left ( \sqrt{\frac{1}{\log g}}+c(g) \right)\sqrt{b+1}.$$
      We note that $c(g)\leq \sqrt{\frac{1}{\log g}}$. Combining with Theorem \ref{uniweyl}, we have:
      \begin{equation}\label{eq36}
        \begin{aligned}
       (-C_{0}-2C_{0}^{\prime})\sqrt{\frac{b+1}{\log g}} 
       &\leq \frac{N^{-\Delta_{\omega}}_{X}(\Re \lambda \in I)}{\mathrm{Vol}(X)}-\frac{1}{4\pi}\int^{\infty}_{\frac{1}{4}} \mathbf{1}_{[a,b]}(\lambda)\tanh \left ( \pi\sqrt{\lambda-\frac{1}{4}} \right ) d\lambda \\
       &\leq C_{0}\sqrt{\frac{b+1}{\log g}} \left[ \log \left(2+(b-a) \sqrt{\frac{\log g}{b+1}}\right) \right]^{\frac{1}{2}} + 2C_{0}^{\prime}\sqrt{\frac{b+1}{\log g}} \\
       &\leq (C_{0}+4C_{0}^{\prime})\sqrt{\frac{b+1}{\log g}} \left[ \log \left(2+(b-a) \sqrt{\frac{\log g}{b+1}}\right) \right]^{\frac{1}{2}}.
        \end{aligned}
      \end{equation}
      It implies \eqref{maincount} and \eqref{maincountfin} in the case $b^{\prime}-a^{\prime}\geq 2\varepsilon\sqrt{\frac{1}{\log g}}h$ by $C=10(C_{0}+C_{0}^{\prime})$. \\
      \indent If $b^{\prime}-a^{\prime}\leq 2\varepsilon \sqrt{\frac{1}{\log g}}h$, let $a_{0}=b-2\varepsilon\sqrt{\frac{b+1}{\log g}} \geq -2$. On one side, 
      \begin{equation*}
        \begin{aligned}
       &\frac{N^{-\Delta_{\omega}}_{X}(\Re \lambda \in I)}{\mathrm{Vol}(X)}-\frac{1}{4\pi}\int^{\infty}_{\frac{1}{4}} \mathbf{1}_{[a,b]}(\lambda)\tanh \left ( \pi\sqrt{\lambda-\frac{1}{4}} \right ) d\lambda \\
       &\geq -\frac{b-a}{4\pi}\geq -\frac{2\varepsilon}{4\pi} \sqrt{\frac{b+1}{\log g}}> -C\sqrt{\frac{b+1}{\log g}}.
        \end{aligned}
      \end{equation*}
      On the other side, by \eqref{eq36} and $\varepsilon=\min \left\{C_{0}, \ C_{0}^{\prime}, \ 1\right\}$, we have:
      \begin{equation*}
        \begin{aligned}
       \frac{N^{-\Delta_{\omega}}_{X}(\Re \lambda \in [a,b])}{\mathrm{Vol}(X)} 
       &\leq  \frac{N^{-\Delta_{\omega}}_{X}(\Re \lambda \in [a_{0},b])}{\mathrm{Vol}(X)} \\
       &\leq \frac{2\varepsilon}{4\pi}\sqrt{\frac{b+1}{\log g}}+ (C_{0}+4C_{0}^{\prime})\sqrt{\frac{b+1}{\log g}} \left( \log \left(2+2\varepsilon\right) \right)^{\frac{1}{2}} \\
       &\leq \left( 2C_{0}\sqrt{\frac{b+1}{\log g}}+ (2C_{0}+8C_{0}^{\prime})\sqrt{\frac{b+1}{\log g}}\right)\left( \log 2 \right)^{\frac{1}{2}} \\
       &\leq  C \sqrt{\frac{b+1}{\log g}} \left[ \log \left(2+(b-a) \sqrt{\frac{\log g}{b+1}}\right) \right]^{\frac{1}{2}}.
        \end{aligned}
      \end{equation*}
      So we complete the proof of our main Theorem \ref{uniweylfortwist}.
    \subsection{Proof of Theorem \ref{uniweylfortwistsmalleigen} and \ref{twigap}}
    First, we review Wu-Xue \cite[Theorem 2]{WX2}.
    \begin{theorem}[Wu-Xue, \cite{WX2}]\label{smalleigen}
For any $\sigma \in\left(\frac{1}{2}, 1\right)$, and for any $\epsilon>0$, we have
$$
\lim _{g \to \infty} \mathbb{P}^{\mathrm{WP}}_{g}\left(X \in \mathcal{M}_g ; N_{X}^{-\Delta}(\lambda\in (0,\sigma(1-\sigma))) \leq g^{3-4 \sigma+\epsilon}\right)=1.
$$
    \end{theorem}
    Now we start the proof. Fix an interval $I=[a,b]$ with $0\leq b<\frac{1}{4}$, for any $0<\varepsilon<\sqrt{\frac{1}{4}-b}$, and $c>0$, let $\sigma=\frac{1}{2}+\sqrt{\frac{1}{4}-b-\varepsilon^{2}} \in (0,1)$, then $\sigma(1-\sigma)=b+\varepsilon^{2}$. We introduce a sequence of events
    \begin{equation}\label{smalleigencount}
      \mathcal{N}^{b,\varepsilon}_{g}:=\left\{X \in \mathcal{M}_{g}: N_{X}^{-\Delta}(\lambda\in (0,\sigma(1-\sigma))) \leq g^{1-4\sqrt{\frac{1}{4}-b-\varepsilon^{2}}+\frac{\varepsilon}{2}} \right\},
    \end{equation}
    which occurs with high probability by Theorem \ref{smalleigen}. Let $$g_{c,\varepsilon}:=\left[g_{c}+(4000c \varepsilon^{-2})^{48}+(1+C_{0}^{\prime\prime})^{\frac{2}{\varepsilon}}\right]+1,$$ 
    where $g_{c}$ is defined in \eqref{eq37}. By \eqref{eq31}, for any $g>g_{c,\varepsilon}$, for any $X \in \mathcal{A}_{g} \cap \mathcal{L}_{g}\cap \mathcal{N}^{b,\varepsilon}_{g}$ and $\omega \in \mathcal{H}^{1}(X,\mathbb{R})$ with $||\omega||_{L^2}\leq c$, we have $||\omega||_{L^\infty}\leq c(g)=cg^{-\frac{1}{48}}$ and $c(g)<10^{-4}$. By \eqref{eq38}, $\lambda_{0}^{\omega}(X)>-\frac{1}{2}$, we have: $$N^{-\Delta_{\omega}}_{X}(\Re \lambda \in [a,b]) \leq N^{-\Delta_{\omega}}_{X}(\Re \lambda \in [-1,b]),$$ since $c(g)<10^{-4}$ and $g>(4000c \varepsilon^{-2})^{48}$, then $$b-(-1)\geq 1> 2\delta=8000c(g), \quad -1+\delta<-1/2, \quad b+\delta<b+\varepsilon^2.$$
      Since $X \in \mathcal{N}^{b,\varepsilon}_{g} $, by \eqref{smalleigencount}, we have:
      $$N^{-\Delta}_{X}(-1+(-\delta,\delta))=0, \ N^{-\Delta}_{X}(b+(-\delta,\delta))\leq  N_{X}^{-\Delta}(\lambda\in (0,\sigma(1-\sigma))) \leq g^{1-4\sqrt{\frac{1}{4}-b-\varepsilon^{2}}+\frac{\varepsilon}{2}},$$
      it implies that $$N^{P}_{X}(\lambda \in I_{\delta}) \leq g^{1-4\sqrt{\frac{1}{4}-b-\varepsilon^{2}}+\frac{\varepsilon}{2}}.$$
      When $h=1$, $g$ and the interval $[-1,b]$ satisfy \eqref{eq46} and \eqref{eq33} in Theorem \ref{perturoflaplacian} respectively, thus by Theorem \ref{perturoflaplacian}, we obtain that:
      \begin{equation}
        \left | N^{-\Delta_{\omega}}_{X}(\Re \lambda \in [-1,b])-N^{-\Delta}_{X}(\lambda \in [-1,b]) \right |\leq C_{0}^{\prime\prime} g^{1-4\sqrt{\frac{1}{4}-b-\varepsilon^{2}}+\frac{\varepsilon}{2}}.
      \end{equation}
      By Theorem \ref{smalleigen}, we have $$N^{-\Delta}_{X}(\lambda \in [-1,b])\leq 1+ g^{1-4\sqrt{\frac{1}{4}-b-\varepsilon^{2}}+\frac{\varepsilon}{2}}.$$
      Since $g>(1+C_{0}^{\prime\prime})^{\frac{2}{\varepsilon}}$, we have:
      $$N^{-\Delta_{\omega}}_{X}(\Re \lambda \in [a,b]) \leq N^{-\Delta_{\omega}}_{X}(\Re \lambda \in [-1,b])\leq 1+ g^{1-4\sqrt{\frac{1}{4}-b}+\varepsilon}.$$
      We complete the proof of Theorem \ref{uniweylfortwistsmalleigen}. \\
      \indent The proof of Theorem \ref{twigap} is very similar to Theorem \ref{uniweylfortwistsmalleigen}. We select a typical spectral gap $\alpha \in \left(0,\frac{1}{4}\right]$, any $0<\varepsilon<\alpha$, and any $c>0$, let $$g_{\alpha,c,\varepsilon}=\left[g_{c}+(16000c \varepsilon^{-1})^{48}+(4000c(\alpha-\varepsilon)^{-1})^{48}\right]+1,$$ 
        where $g_{c}$ is defined in \eqref{eq37}. By \eqref{eq31}, for any $g>g_{\alpha,c,\varepsilon}$, for any $X \in \mathcal{A}_{g} \cap \mathcal{L}_{g} \cap \mathcal{Q}^{\alpha,\varepsilon/2}_{g}$, where $\mathcal{Q}^{\alpha,\varepsilon}_{g}$ is defined in \eqref{gapevent}, and $\omega \in \mathcal{H}^{1}(X,\mathbb{R})$ with $||\omega||_{L^2}\leq c$, we have $$||\omega||_{L^\infty}\leq c(g)=cg^{-\frac{1}{48}}, \quad c(g)<10^{-4}.$$ By \eqref{eq38}, $\lambda_{0}^{\omega}(X)>-\frac{1}{2}$, we choose $a=-1$ and $b=\alpha-\varepsilon$ and $\delta=4000c(g)$, then by $c(g)<10^{-4}$, $g>(16000c \varepsilon^{-1})^{48}$ and $g>(4000c(\alpha-\varepsilon)^{-1})^{48}$, we have:
        $$a+\delta<0<b-\delta<b+\delta<b+\frac{\varepsilon}{4}<\alpha-\frac{\varepsilon}{2}.$$ 
        By \eqref{gapevent}, $\lambda_{1}(X)>\alpha -\frac{\varepsilon}{2}$, thus we have $N^{-\Delta}_{X}(\lambda \in I_{\delta})=0$ and $N^{-\Delta}_{X}(\lambda \in [a,b])=1$. 
        When $h=1$, $g$ and the interval $[a,b]$ satisfy \eqref{eq46} and \eqref{eq33} in Theorem \ref{perturoflaplacian} respectively, thus by Theorem \ref{perturoflaplacian}, we obtain that:
         \begin{equation}
           \left | N^{-\Delta_{\omega}}_{X}(\Re \lambda \in [a,b])-N^{-\Delta}_{X}(\lambda \in [a,b]) \right |=0,
         \end{equation}
         so we obtain $ N^{-\Delta_{\omega}}_{X}(\Re \lambda \in [a,b])=1$. By $a<\lambda_{0}^{\omega}(X)\leq 0<b=\alpha-\varepsilon$, then $\Re \lambda^{\omega}_{1}>\alpha-\varepsilon.$ We complete the proof of Theorem \ref{twigap}.

\section{Appendix: $L^{\infty}$ estimate of eigenforms}\label{app2}

\subsection{Untempered Eigenforms}
First, we discuss the untempered eigenforms $\omega$ which satisfies:
\begin{equation*}
  \Delta \omega=\left ( \frac{1}{4}+r^{2} \right ) \omega, \ \text{where} \ r=is, s\in \left[ \beta, \frac{1}{2} \right).
\end{equation*}
We use the test function $k_{t}(x,y)$ in \eqref{untermpertest}. By Le Masson-Sahlsten \cite[Lemma 8.1]{LS}, we have:
\begin{equation*}
  \begin{aligned}
  h_{t}(r)& =\mathcal{S}(k)(r)=2\sqrt{2}\int_{0}^{t}\cos(ru)\sqrt{1-\frac{\cosh(u)}{\cosh(t)}}du \\
          & =2\sqrt{2}\int_{0}^{t}\cosh(su)\sqrt{1-\frac{\cosh(u)}{\cosh(t)}}du.
  \end{aligned}
\end{equation*}
Now we estimate the $h_{t}(is)=h_{t}(r)$,
\begin{equation}
  \begin{aligned}
    h_{t}(is) & \gtrsim \int_{0}^{t}\cosh(su)\sqrt{1-\frac{\cosh(u)}{\cosh(t)}}du \geq \int_{0}^{t}\cosh(su)\left ( 1-\frac{\cosh(u)}{\cosh(t)} \right )du\\
              & =\frac{\sinh(st)}{s}-\frac{\sinh((s+1)t)}{2\cosh(t)(s+1)}-\frac{\sinh((1-s)t)}{2\cosh(t)(1-s)}\gtrsim_{\beta} \sinh(\beta t).
  \end{aligned}
\end{equation}
It implies:
\begin{equation}\label{untermperlinf}
  |A\omega|_{x}=|h_{t}(is)||\omega|_{x}\gtrsim_{\beta} \sinh(\beta t) |\omega|_{x}.
\end{equation}
Another side, for the untempered eigenform $\omega=df-\star dg$, we have:
\begin{equation*}
    (A\omega)_{x}=(dF_{1})_{x}-\star (dF_{2})_{x}.
\end{equation*}
where
\begin{equation*}
    F_{1}(x)=\int_{\mathbb{H}}k_{t}(x,y)f(y)d\mu(y), \quad F_{2}(x)=\int_{\mathbb{H}}k_{t}(x,y)g(y)d\mu(y).
\end{equation*}
By \eqref{4geo}, we take $R=(\frac{1}{2}-\epsilon)\log g$ and $t=\frac{R}{2}$ in \eqref{count2} then we obtain that:
\begin{equation*}
  \begin{gathered}
    |dF_{1}|_{x}\leq \sqrt{2\pi \left ( \frac{4t}{\mathrm{Inj}(X)}+2 \right )} ||df||_{L^{2}}, \quad |dF_{2}|_{x}\leq \sqrt{2\pi \left ( \frac{4t}{\mathrm{Inj}(X)}+2 \right )} ||dg||_{L^{2}}. 
  \end{gathered}
\end{equation*}
It implies that for sufficiently large $g$\footnote{Hodge star operator $\star$ is an isomrphism between $T^*_{x}X$ for all $x\in X$}:
\begin{equation}\label{untermperl2}
  \begin{aligned}
  |A\omega|_{x} & \leq |dF_{1}|_x+|\star (dF_{2})|_x \leq |dF_{1}|_x+|dF_{2}|_x \\
                & \leq  \sqrt{2\pi \left ( \frac{4t}{\mathrm{Inj}(X)}+2 \right )}(||df||_{L^{2}}+||dg||_{L^{2}}) \leq 2\sqrt{\frac{6\pi \log g}{\mathrm{Inj}(X)}}||\omega||_{L^{2}}. \\
  \end{aligned}
\end{equation}
We obtain the following consequence by \eqref{untermperlinf} and \eqref{untermperl2}:
\begin{equation*}
  \sinh(\beta t)|\omega|_{x}\lesssim \sqrt{\frac{\log g}{\mathrm{Inj}(X)}}||\omega||_{L^{2}}.
\end{equation*}
It implies that for sufficient large $g>g_{\epsilon,\beta}$:
\begin{equation*}
  ||\omega||_{L^{\infty}}\lesssim \frac{g^{-(1/2-\epsilon)\beta/2}\log g}{\sqrt{\mathrm{Inj}(X)}}||\omega||_{L^{2}}\leq \frac{g^{-\frac{\beta}{4}+\epsilon}}{\sqrt{\mathrm{Inj}(X)}}||\omega||_{L^{2}}.
\end{equation*}
We notice that it coincides with Theorem \ref{deh}, but the construction of the operator on the harmonic forms and untempered eigenforms are different so we need to prove two theorems respectively. Now we proved the tempered cases.
\subsection{Tempered Eigenforms}
First, we mention that the operator $A$ which is generated by radial kernel $k$ on the function is denoted by the same notation as the operator generated by the same kernel on the 1-form. In this subsection, we denote $||\cdot||$ as the $L^2$ norm on $\Omega^{1}(X)$, and $||\cdot||_{L^2(X)}$ as the $L^2$ norm on $L^2(X)$.\\
\indent The tempered eigenform $\omega$ satisfies:
\begin{equation*}
  \Delta \omega=\left ( \frac{1}{4}+r^{2} \right )\omega, \ r \in \mathbb{R}. 
\end{equation*}
We choose the auxiliary functions in Gilmore-Le Masson-Sahlsten-Thomas \cite[Section 3.2]{GCLS}:
\begin{equation}\label{auxiliary}
  \begin{aligned}
   j_{t}(r):=\frac{\cos(rt)}{\sqrt{\cosh(\frac{\pi r}{2})}}, \quad l_{t}(\rho):=\mathcal{S}^{-1}(j_{t}).
  \end{aligned}
\end{equation}
We define $P_{t}$ as the operator on 1-forms generated by $l_{t}(\rho)$. When $P_t$ acts on the functions on $X$, for conjugate indices $p,q$ with $p>2$, Gilmore-Le Masson-Sahlsten-Thomas proved the following lemma.
\begin{lemma}[Gilmore-Le Masson-Sahlsten-Thomas, \cite{GCLS}]\label{decomposelemma}
The integral kernel of the composition operator $P_t P_s^*: L^q(X) \rightarrow L^p(X)$ for $t, s \geq 0$ is given by
$$
\frac{1}{2}\left(k_{t+s}+k_{|t-s|}\right)
$$
where $k_{t}(x,y)=K_{t}(d(x,y))$ is the associated radial kernel through the inverse Selberg transform with the function
\begin{equation}\label{eq51}
h_t(r)=\frac{\cos (r t)}{\cosh \left(\frac{\pi r}{2}\right)}, \quad K_t=\mathcal{S}^{-1}(h_t).
\end{equation}
In particular, if $Q_t: L^q(X) \rightarrow L^p(X)$ is the associated integral operator for the kernel $k_t$, then
$$
P_t P_s^* =\frac{1}{2}\left(Q_{t+s} +Q_{|t-s|} \right) .
$$
\end{lemma}
The function $h_{t}(r)$ is considered by Brooks and Lindenstrauss in \cite{BSLE}, and introduced by Iwaniec and Sarnak in \cite{IWH21}, which is used to estimate the $L^{\infty}$ norm of eigenfunctions on the Arithmetic surfaces. Here we need the following Lemma \cite[Lemma 5.2]{BSLE}, also see \cite[Lemma 3.2]{GCLS}:
\begin{lemma}[Brooks-Lindenstrauss, \cite{BSLE}]\label{lemBL}
  We have the following estimate for $K_{t}$ in \eqref{eq51}, a sup norm bound of 
  \begin{equation*}
    \sup_{\rho\geq 0}|K_{t}(\rho)|\lesssim e^{-\frac{t}{2}},
  \end{equation*}
  and rapid decay outside a ball of radius $4t$ of the type
  \begin{equation*}
    \int_{4t}^{\infty}|K_{t}(\rho)|\sinh(\rho)d\rho\lesssim e^{-t}.
  \end{equation*}
\end{lemma}
Now we use Lemma \ref{lemBL} to conclude that for any $X\in \mathcal{L}_{g}$, we take $R=\frac{1}{4}\log g$ and $t \leq \frac{R}{8}$, the arguement in \cite[Page 13]{GCLS} and Lemma \ref{lemBL} show that:
\begin{equation*}
  \begin{aligned}
    \sum_{\gamma \in \Gamma}|k_{t}(x,\gamma y)| 
    & \lesssim \sum_{\gamma \in \Gamma}|k_{t}(x,\gamma y)|\mathbf{1}_{[0,4t]}(d(x,\gamma y))+\int_{4t}^{\infty}K_{t}(\rho)e^{\rho}d\rho \\
    &\lesssim \# \{ \gamma \in \Gamma: d(x,\gamma y)\leq 4t\}e^{-\frac{t}{2}}+e^{-t} \lesssim \frac{t}{\mathrm{Inj}(X)}e^{-\frac{t}{2}}.
  \end{aligned}
\end{equation*}
It implies $Q_{t}: L^{1}(X) \to L^{\infty}(X)$ satisfies:
\begin{equation}\label{eq39}
  ||Q_{t}||_{L^{1}(X)\to L^{\infty}(X)}\lesssim \frac{t}{\mathrm{Inj}(X)}e^{-\frac{t}{2}}.
\end{equation}  
Now we introduce the following operator:
for any $X \in \mathcal{L}_{g}$, take $R=\frac{1}{4}\log g$ and fixed $T\leq \frac{R}{8}$. We define $W_{T,\lambda}: \Omega^{1}(X) \to \Omega^{1}(X)$ as:
\begin{equation*}
  (W_{T,\lambda}\omega)_{x}=\int_{0}^{T} \cos(r_{\lambda}t)(P_{t}\omega)_{x}dt, \ r_{\lambda} \in \mathbb{R}.
\end{equation*}
Here we need the following lemma:
\begin{lemma}\label{ineqonform}
  For any operator $T$ which is generated by radial function $a\geq 0$, for any $\varphi \in C^{\infty}(X)$ thus we have:
  \begin{equation*}\label{eq34}
    |T(d\varphi)|_{x} \leq T(|d\varphi|)(x), \quad \text{ for all } \ x\in X.
  \end{equation*}
\end{lemma}
\begin{proof}
  We consider $\varphi$ as $\Gamma$-invariant smooth function on $\mathbb{H}$, for any $v \in T_{x}\mathbb{H}$ with $|v|_{x}=1$, by \eqref{operatorform} and \eqref{radialfunction}, we have:
  $$\langle T(d\varphi),v\rangle_{x}=d\Phi_{x}(v), \quad \Phi(x)=\int_{\mathbb{H}}a(x,y)\varphi(y)d\Vol(y).$$ 
  By the same argument of \eqref{calculate1}, we assume that $x=i, v=(0,1)$, then 
  $$d\Phi_{x}(v)=\int_{\mathbb{H}}a(x,y)(d\varphi)_{y}(y)d\Vol(y)\leq \int_{\mathbb{H}}a(x,y)|d\varphi|_{y}d\Vol(y)=T(|d\varphi|)(x).$$ 
  It implies that $$|T(d\varphi)|_{x}=\sup_{v \in T_{x}\mathbb{H}, |v|_{x}=1}\langle T(d\varphi),v\rangle_{x}\leq T(|d\varphi|)(x).$$
\end{proof}
Then by \eqref{diffinteresti} and Lemma \ref{ineqonform}, we have:
\begin{equation}\label{eq41}
  \begin{aligned}
    |W_{T,\lambda}(df)|_x=\left | \int_{0}^{T} \cos(r_{\lambda}t)P_{t}(df) dt \right |_x \leq \int_{0}^{T} |P_{t}(df)|_{x}dt \leq \int_{0}^{T} P_{t}(|df|)(x)dt.
  \end{aligned}
\end{equation}
By \eqref{eq39}, for any $\delta>0$, and for sufficient large $g$, we have:
\begin{equation}\label{eq40}
  ||Q_{t}||_{L^{1}(X)\to L^{\infty}(X)}\lesssim \frac{e^{-\left( \frac{1}{2}-\delta \right)t}}{\mathrm{Inj}(X)}.
\end{equation}
Then we use the $AA^{*}$ argument,
\begin{equation*}
  \begin{aligned}
   \left \Vert \int_{0}^{T} P_{t}(|df|)dt \right \Vert_{L^{\infty}}&\leq \left  \Vert \int_{0}^{T} P_{t} dt \right \Vert_{L^{2}(X)\to L^{\infty}(X)} ||df||_{L^{2}} \\
                                              &\leq \left \Vert \left ( \int_{0}^{T} P_{t} dt \right ) \left (\int_{0}^{T} P_{t} dt \right )^{*} \right \Vert_{L^{1}(X)\to L^{\infty}(X)}^{1/2} ||df||_{L^{2}} \\
                                              &\leq \left \Vert \int_{0}^{T}\int_{0}^{T} P_{t}P_{s}^{*}dtds \right \Vert_{L^{1}(X)\to L^{\infty}(X)}^{1/2} ||df||_{L^{2}}.
  \end{aligned}
\end{equation*}
By Lemma \ref{decomposelemma} and \eqref{eq40},
\begin{equation}\label{eq42}
  \begin{aligned}
    &\left \Vert \int_{0}^{T}\int_{0}^{T} P_{t}P_{s}^{*}dtds \right \Vert_{L^{1}(X)\to L^{\infty}(X)} \\
    &\lesssim  \left \Vert \int_{0}^{T}\int_{0}^{T} Q_{t+s} dtds \right \Vert_{L^{1}(X)\to L^{\infty}(X)}+\left \Vert \int_{0}^{T}\int_{0}^{T} Q_{|t-s|}dtds \right \Vert_{L^{1}(X)\to L^{\infty}(X)}\\
    &\lesssim  \frac{1}{\mathrm{Inj}(X)}\int_{0}^{T}\int_{0}^{T}e^{-\left( \frac{1}{2}-\delta \right)(t+s)}+e^{-\left( \frac{1}{2}-\delta \right)|t-s|}dtds \lesssim \frac{T}{\mathrm{Inj}(X)}.
  \end{aligned}
\end{equation}
By \eqref{eq41} and \eqref{eq42}, we prove that:
\begin{equation*}
  |W_{T,\lambda}(df)|_{x} \lesssim \sqrt{\frac{T}{\mathrm{Inj}(X)}} ||df||_{L^{2}}.
\end{equation*}
It also holds for $\star dg$, so we have:
\begin{equation}\label{eq45}
  |W_{T,\lambda}\omega|_{x} \lesssim \sqrt{\frac{T}{\mathrm{Inj}(X)}} ||\omega||_{L^{2}}.
\end{equation}
Now we consider the eigenform $\omega=df-\star dg$ associated with the eigenvalue $\lambda$ where $\lambda=\frac{1}{4}+r^{2}, r\in \mathbb{R}.$
Then 
\begin{equation}\label{eq43}
  |W_{T,\lambda}\omega|_{x}=\left |\frac{1}{\sqrt{\cosh(\frac{\pi r}{2})}}\int_{0}^{T}\cos^{2}(rt)dt \cdot \omega \right |_{x}\gtrsim_{\lambda} T|\omega|_{x}. 
\end{equation}
By \eqref{eq45} and\eqref{eq43}, it implies that:
\begin{equation}
T|\omega|_{x} \lesssim_{\lambda} |W_{T,\lambda}\omega|_{x} \lesssim_{\lambda} \sqrt{\frac{T}{\mathrm{Inj}(X)}} \cdot (||df||_{L^{2}}+||dg||_{L^{2}}  ) \lesssim_{\lambda} \sqrt{\frac{T}{\mathrm{Inj}(X)}} \cdot ||\omega||_{L^{2}}.
\end{equation}
Since $T =\frac{\log g}{32}$, we prove \eqref{termperform}.

\bmhead{Acknowledgments}

The author would like to thank Long Jin for his guidance and encouragement, Yunhui Wu, Xiaolong Han, Yuhao Xue for their valuable suggestions on hyperbolic geometry, including the use of the event $\mathcal{L}_{g}$ to obtain a more explicit upper bound in Theorem \ref{deh}, Joe Thomas for his suggestions on the proof of Theorem \ref{count11}, Zuoqin Wang and Qiaochu Ma for numerous useful discussions. I am also grateful to an anonymous referee for many useful suggestions to improve this article. Yulin Gong is supported by National Key R \& D Program of China 2022YFA100740.  

\section*{Declarations}
\begin{itemize}
\item Conflict of interest. The author declares that there is no conflict of interest. \\
\item Data avaliability. Data sharing is not applicable to this article as no datasets were generated or analyzed during the current study.
\end{itemize}


\bibliography{sn-ref}

\end{document}